\documentclass{amsart}

\usepackage{graphics} % for pdf, bitmapped graphics files
\usepackage{epsfig} % for postscript graphics files
\usepackage{times} % assumes new font selection scheme installed
\usepackage{amsmath} % assumes amsmath package installed
\usepackage{amssymb}  % assumes amsmath package installed
\usepackage{cases}
\usepackage{float}
\newtheorem{theorem}{Theorem}[section]
\newtheorem{lemma}{Lemma}[section]%[theorem]
\newtheorem{corollary}{Corollary}[section]%[theorem]
\newtheorem{proposition}{Proposition}[section]
\newtheorem{definition}{Definition}[section]

\newtheorem{remark}{Remark}[section]
\newtheorem{assumption}{Assumption}

\def\cal{\mathcal}

\numberwithin{equation}{section}

\usepackage{mathtools}
\usepackage{latexsym}
\usepackage{dsfont}
\usepackage{graphicx}
\usepackage{float}
\usepackage{mathrsfs}

\usepackage{verbatim}
\usepackage{epstopdf}
\usepackage{url}
\usepackage[utf8]{inputenc}
\usepackage{hyperref}

\DeclareMathOperator\diag{diag}
\DeclareMathOperator\Tr{Tr}

\begin{document}

\title{Linear quadratic mean field games: Decentralized $O(1/N)$-Nash equilibria}

\author{Minyi Huang}
\address{School  of Mathematics and Statistics, Carleton University,
Ottawa, ON K1S 5B6, Canada}
%\curraddr{Department of Mathematics and Statistics,
%Case Western Reserve University, Cleveland, Ohio 43403}
\email{mhuang@math.carleton.ca}
\thanks{In honor of the 60th birthday of Professor Lei Guo, Institute of Systems Science, Chinese Academy of Sciences, Beijing, China.}

 %Information for second author
\author{Xuwei Yang}
\address{School of Mathematics and Statistics, Carleton University, Ottawa, ON K1S 5B6, Canada}
\email{xuweiyang@cunet.carleton.ca}

%    General info
% \subjclass[2000]{XXXX}
%\date{Oct 17, 2019.}

\begin{abstract}
This paper studies an asymptotic solvability problem for  linear quadratic (LQ) mean field games with controlled diffusions and indefinite weights for the state and control in the costs.
We employ a rescaling approach to derive a low dimensional Riccati ordinary differential equation (ODE) system, which characterizes a necessary and sufficient condition  for  asymptotic solvability. The rescaling technique is further used for performance estimates, establishing an $O(1/N)$-Nash equilibrium for the obtained decentralized strategies.
\end{abstract}

\maketitle

\tableofcontents

\section{Introduction}
\label{sec:intro}

Since its inception \cite{HMC2006,LL2007}, mean field game theory has undergone a phenomenal growth and found applications in diverse areas
\cite{BTB2016,CFS2015,CS2017,DAS2016,HJN2019,LW2011,LZ2019,LT2015,LRS2019,MCH2013,
SMN2018,SGU2016,WH2019,YMMS2012}.
The theory is inspired by ideas in statistical physics and overcomes the dimensionality difficulty in competitive decision problems involving a large population of agents. The reader is referred to \cite{BFY2013,CHM2017,C2013,CD2018} for an overview of basic theory and applications.

While mean field games have been developed with very different modelling frameworks, linear quadratic (LQ) mean field games are of particular importance and have been extensively studied due to their elegant closed-form solutions  \cite{BSYY2016,  HWW2016,  HCM2007,SMN2018}. Huang, Caines and Malham\'e \cite{HCM2007} adopt infinite horizon discounted costs and use the infinite population limit model to design decentralized strategies for the actual model with a large but finite population.
Li and Zhang \cite{LZ2008} study decentralized strategies with ergodic costs.
Wang and Zhang \cite{WZ2012} introduce Markov jumps in the system dynamics and costs.
Bardi and Priuli \cite{BP2014} study LQ $N$-person games and their mean field limit with ergodic costs. Huang, Wang and Wu \cite{HWW2016} adopt backward stochastic differential equations for modelling state processes. Moon and Basar \cite{MB2017} consider risk sensitive costs and address robustness.
Huang and Huang \cite{HH2017} consider linear diffusion dynamics
including model uncertainty treated as an adversarial player.
Tchuendom \cite{T2018} shows nonuniqueness can arise, but interestingly, uniqueness can be restored by the presence of common noise.
LQ mean field games have an extension by including a major player \cite{H2010,NH2012}.
 This modelling framework is introduced by Huang \cite{H2010}.
Bensoussan et al \cite{BCLY2017} consider Stackelberg equilibria under state and control delays. Caines and Kizikale \cite{CK2017} consider partial information and filtering based strategies for an LQ model with a major player.

In this paper we study a class of LQ mean field games
with common noise and indefinite weight matrices (simply called weights below) in the cost functional. We adopt the so-called asymptotic solvability framework in \cite{HZ2020}. Starting with  feedback perfect state information, this approach aims to determine feedback Nash strategies under such centralized information and next study how the solutions behave when the number of players increases.
It uses a rescaling method to derive a set of  Riccati ordinary differential equations (ODEs), which characterizes a necessary and sufficient condition for asymptotic solvability \cite{HZ2020}. This method can be extended to LQ mean field games with a major player \cite{MH2020}. Recently, Huang and Yang \cite{HY2020b} extend this asymptotic solvability notion to mean field social optimization, where the agents cooperatively optimize a social cost. That work further develops  a method of asymptotic analysis  to obtain tight estimates of optimality loss
when decentralized strategies are implemented. For our current model,  the test of asymptotic solvability reduces to checking two Riccati ODEs in a low dimensional space, which, as a result of the controlled individual and common noises, have higher nonlinearity than those Riccati
equations in \cite{HZ2020}.

In the analysis of mean field games, a crucial step is to examine how the strategies obtained in the mean field limit model perform when
implemented in the actual model with a large but finite population.
This can be addressed by establishing the so-called $\epsilon$-Nash equilibrium property, where $\epsilon\to 0$ as $N\to \infty$.
For LQ models \cite{BSYY2016,HWW2016,HCM2007,WZ2012} as well as some nonlinear cases \cite{NC2013}, one can obtain an $O(1/\sqrt{N})$-Nash equilibrium when all players are symmetric. This typically results from  cost estimates by the Cauchy-Schwarz inequality.
To our best knowledge on the existing literature, probably only  Basna, Hilbert and  Kolokoltsov \cite{BHK2014} have obtained an $O(1/N)$-Nash equilibrium result in a finite state mean field game. We will establish an $O(1/N)$-Nash equilibrium for the decentralized strategies obtained from the LQ mean field limit model;  our approach is  different from that in \cite{BHK2014} which relies on perturbation estimates of generators of continuous-time controlled Markov chains. We will
 directly treat the best response control problem of the unilateral agent in a high dimensional space and then employ the rescaling method to obtain accurate information about its performance improvement.
 We will develop extensive asymptotic error estimates by building upon techniques in the companion paper \cite{HY2020b} on social optimization.
  In a convergence problem of mean field games with common noise, Cardaliaguet et al \cite{CDLL2015} prove that the value functions of $N$ players converge in an average sense to the solution of the master equation, and the averaged error disappears by rate $1/N$ as $N\to \infty$. But their error bound is different from the $O(1/N)$-Nash equilibrium notion.

It will be helpful to briefly explain the route that we will follow in the analysis.
For the  LQ  Nash game with indefinite weights, we apply dynamic programming to derive a set of large-scale Riccati equations, which is used to
formulate the asymptotic solvability problem of the $N$-player game. In order to get useful information from the large Riccati equations, we exploit their symmetries to achieve dimension reduction and   next use a rescaling technique to  derive two key Riccati equations, which completely  characterize asymptotic solvability. By taking the mean field limit of the solution of the $N$-player game, we construct a set of decentralized strategies, which are then applied  to the $N$-player model.
We further obtain  explicit formulas for the per agent cost for three scenarios: i) the $N$ players apply the Nash equilibrium strategies $(\hat u_1, \cdots, \hat u_N)$; ii) the $N$ players apply  decentralized strategies $(\check u_1, \cdots, \check u_N)$ obtained from the mean field limit model; iii)  the player in question takes its best response while the other $N-1$ players apply these decentralized strategies.
When $N\to \infty$, the three cases have the same limit for the per agent cost. The comparison of the costs in scenarios ii) and iii) establishes the  $O(1/N)$-Nash equilibrium property. A comparison of the per agent costs
for the mean field game and the mean field social optimum enables us to  quantify the efficiency loss of the mean field game with respect to the social optimum; see the comparison in the companion paper \cite{HY2020b}.

%%%%%%%%%%%%%
%%%%%%%%%%%%%%%%
\subsection{Organization of the paper}

Section \ref{sec:LQNG} introduces the $N$-player LQ Nash game with indefinite weight matrices in the cost functionals. The set of feedback Nash equilibrium strategies is characterized using a system of Riccati ODEs in Section
 \ref{sec:RiccatiEqns}.
The asymptotic solvability problem is studied in  Section \ref{sec:AS} and a necessary and sufficient condition is derived. Section \ref{sec:DecentralStrats} constructs a set of decentralized strategies for
the $N$-player game, and
 Section \ref{sec:O(1/N)Nash} proves an $O(1/N)$-Nash equilibrium theorem.
A numerical example is presented in Section \ref{sec:num}.
 Section \ref{sec:conclusion} concludes the paper.

%%%%%%%%%%%%%
\subsection{Notation}

Let ${\mathcal S}^n$ be the set of $n\times n$ real symmetric matrices. We denote the quadratic form $[x]_M^2 = x^T M x$ for $M\in {\mathcal S}^n$ and $x\in \mathbb{R}^n$.
We use $I$ to denote an identity matrix of compatible dimensions, and sometimes write $I_k$ to indicate the $k\times k$ identity matrix.
We use $0$ to denote either the scalar zero or a zero vector/matrix of compatible dimensions.

We denote  by $|F|$  the Euclidean norm of a vector or matrix $F$, by  $\mathbf{1}_{k\times l}$ a $k\times l$ matrix with all entries equal to $1$,  by $\otimes$ the Kronecker product, and by the column vectors $\{e^k_1, \cdots, e^k_k\}$ the canonical basis of $\mathbb{R}^k$.
For a function $f(t,x)$, we may write partial derivatives  $\partial f/\partial t$ as $\partial_t f$;
$\partial f/\partial x$ as $\partial_x f$;  and
$\partial^2 f/\partial x^2$ as $\partial_x^2 f$.

%%%%%%%%%%%%%%%%%%%%%%%%%%%%
%%%%%%%%%%%%%%%%%%%%%%%%%%%%
\section{The LQ Nash game}
\label{sec:LQNG}

Consider a system of $N$ players (or called agents) denoted by $\mathcal{A}_i$, $1\leq i \leq N$. The state process $X_i(t)$ satisfies the following stochastic differential equation (SDE)
\begin{align}
\label{Xi}
d X_i(t) & = (A X_i(t) + Bu_i(t) + G X^{(N)}(t))  dt
 + (B_1 u_i(t) + D) d W_i(t)   \\
& \quad + ( B_0 u^{(N)}(t) + D_0 ) d W_0(t) , \notag
\end{align}
where we have the state $X_i(t) \in \mathbb{R}^n$, the control
$u_i(t)\in \mathbb{R}^{n_1}$,   the mean field state $X^{(N)} \coloneqq(1/N)\sum_{i=1}^N X_i$ and the control mean field  $u^{(N)} \coloneqq (1/N)\sum_{i=1}^N u_i$.
The initial states $\{X_i(0): 1 \leq i \leq N\}$ are independent with  $\mathbb{E}|X_i(0)|^2<\infty$. The individual noise processes $\{W_i: 1\leq i \leq N\}$ are $1$-dimensional independent standard Brownian motions, which are also independent of $\{X_i(0): 1\leq i \leq N\}$.
The common noise $W_0$ is a $1$-dimensional standard Brownian motion independent of $\{W_i: 1 \leq i \leq N\}$ and $\{X_i(0): 1 \leq i \leq N\}$.
In contrast to \cite{HZ2018a,HZ2020}, each individual noise is affected by that player's control, and the model contains a common noise affected by the control mean field.

The individual cost functional (simply called cost) of  $\cal{A}_i$, $1\leq i \leq N$, is given by
\begin{align}
J_i(u_{1}, \cdots, u_N)  = & \mathbb{E} \left[ \int_0^T \Big( [X_i(t) - \Gamma X^{(N)}(t)]_Q^2 + [u_i(t)]_{R}^2 \Big)  \ dt \right.
 \label{Ji}\\
& \quad \left. +  [X_i(T) - \Gamma_f X^{(N)}(T)]_{Q_f}^2 \right] ,
\notag
\end{align}
where we denote $[x]_M^2 = x^T M x$ for $M\in\mathcal{S}^n$ and $x\in \mathbb{R}^n$.
The constant matrices $A$, $B$, $B_0$ $B_1$, $D$, $D_0$, $G$, $\Gamma$, $Q$, $R$, $\Gamma_f$, $Q_f$ above have compatible dimensions, and $Q$, $Q_f$, $R$ are symmetric, possibly indefinite,  matrices.

Define
\begin{align}
& X(t) = \begin{bmatrix}X_1(t) \\ \vdots \\ X_N(t) \end{bmatrix} \in
\mathbb{R}^{Nn},\quad u_{-i}=(u_1, \cdots, u_{i-1}, u_{i+1}, \cdots, u_N), \notag \\
&  \mathbf{A}
= \diag[A, \cdots, A] + \mathbf{1}_{N \times N} \otimes \tfrac{G}{N} \in \mathbb{R}^{Nn \times Nn} ,  \notag \\
&  \mathbf{B}_0 =  \mathbf{1}_{N\times 1} \otimes  \tfrac{B_0}{N} \in
\mathbb{R}^{Nn \times n_1} , \quad
\mathbf{D}_0 = \mathbf{1}_{N\times 1} \otimes D_0 \in \mathbb{R}^{Nn \times 1} ,   \notag \\
&  \widehat{\mathbf{B}}_k = e^N_k \otimes B  \in \mathbb{R}^{Nn \times n_1} , \quad
 \mathbf{B}_k = e^N_k \otimes B_1 \in \mathbb{R}^{Nn\times n_1} , \nonumber\\
&  \mathbf{D}_k = e^N_k \otimes D \in \mathbb{R}^{Nn \times 1} , \quad 1 \leq k \leq N .   \notag
\end{align}
Then $X =(X_1^T, \cdots, X_N^T)^T$ has the following dynamics
\begin{align} \label{X}
d X(t)  = &  \Big( \mathbf{A} X(t) + \sum_{i=1}^N \widehat{\mathbf{B}}_i u_i(t)\Big) dt + \sum_{i=1}^N ( \mathbf{B}_i u_i(t) + \mathbf{D}_i ) d W_i
\\  & + \Big( \mathbf{B}_0 \sum_{i=1}^N u_i(t)  + \mathbf{D}_0 \Big) d W_0 . \notag
\end{align}

We denote
\begin{align}
& \mathbf{K}_i
= [0, \cdots, 0, I_n, 0, \cdots, 0] - (1/N) [\Gamma, \Gamma, \cdots,
\Gamma]\in \mathbb{R}^{n\times Nn} , \notag \\
& \mathbf{K}_{i f} = [0, \cdots, 0, I_n, 0, \cdots, 0] - (1/N) [\Gamma_f , \Gamma_f, \cdots, \Gamma_f ] ,  \notag \\
& \mathbf{Q}_i = \mathbf{K}_i^T Q \mathbf{K}_i, \quad
\mathbf{Q}_{if} = \mathbf{K}_{i f}^T Q_f \mathbf{K}_{i f} . \notag
\end{align}
The individual cost \eqref{Ji} can be written as
\begin{align}
J_i(u_{i}, u_{-i}) & = \mathbb{E} \left[ \int_0^T \Big( [X(t)]_{\mathbf{Q}_i}^2 + [u_i(t)]_{R}^2 \Big)  \ dt
 +  [X(T) ]_{\mathbf{Q}_{i f}}^2 \right] .
\label{Ji2}
\end{align}
We begin by  solving the LQ Nash game
 under  closed-loop perfect state (CLPS) information,
where the full state vector $X(t)$ is observed by each player.
The players seek a set of Nash equilibrium strategies $(\hat u_1, \cdots, \hat u_N)$.

For notational simplicity, Sections \ref{sec:RiccatiEqns}--\ref{sec:O(1/N)Nash}
will treat  a simplified model \eqref{Xi}--\eqref{Ji} with $D= D_0 = 0$.
 The extension to the general case will be discussed in Section \ref{sec:O(1/N)Nash}.

%%%%%%%%%%%%%%%%%%%%%%%
%%%%%%%%%%%%%%%%%%%%%%%
\section{Riccati equations and feedback Nash strategies}
\label{sec:RiccatiEqns}

Based on \eqref{Ji2}, we may naturally define the cost $J(t,\mathbf{x}, u_1, \cdots, u_N)$ where the running cost is integrated on $[t,T]$ instead of $[0,T]$ with initial state $X(t) = \mathbf{x} = (x_1^T, \cdots, x_N^T)^T$.
Let $V_i(t, \mathbf{x})$ denote the  value function of player $\mathcal{A}_i$.
The Hamilton--Jacobi--Bellman (HJB) equations of the $N$ players associated with \eqref{X}--\eqref{Ji2} (taking $D_0 = D = 0$) are
\begin{align}
\label{HJBV}
 & - \frac{\partial V_i}{\partial t}
 =    \frac{\partial^T V_i}{\partial \mathbf{x}}  \Big( \mathbf{A} \mathbf{x}
+ \sum_{k =1}^N \widehat{\mathbf{B}}_k \hat u_k \Big)
  + \frac{1}{2}  \Big(  \sum_{k=1}^N \hat u_k \Big)^T \mathbf{B}_0^T \frac{\partial^2 V_i}{\partial \mathbf{x}^2 }   \mathbf{B}_0
 \Big( \sum_{k=1}^N \hat u_k \Big)     \\
%%%%%
 &\qquad\qquad   + \frac{1}{2} \sum_{k=1}^N ( \mathbf{B}_k \hat u_k  )^T \frac{\partial^2 V_i}{\partial \mathbf{x}^2}  ( \mathbf{B}_k \hat u_k  )
 + \mathbf{x}^T \mathbf{Q}_i \mathbf{x}
 + \hat u_i^T R \hat u_i   , \notag \\
&  V_i( T, \mathbf{x} ) =  \mathbf{x}^T \mathbf{Q}_{if} \mathbf{x}  , \quad 1\leq i \leq N .   \notag
\end{align}
Each  $\hat u_i$ is the minimizer in the HJB equation of $V_i(t, \mathbf{x})$
as specified below. Taking $(u_1, \cdots, u_N)$ in place of
 $(\hat u_1, \cdots, \hat u_N)$, we write  the right hand side of \eqref{HJBV} in the form:
$${\mathcal H}(\mathbf{x}, \partial_\mathbf{x} V_i, \partial_\mathbf{x}^2 V_i, u_i, u_{-i} ).$$ Then we require
\begin{align}\label{Hminimizer}
\hat u_i= \arg\min_{u_i}  {\mathcal H} (\mathbf{x}, \partial_\mathbf{x} V_i, \partial_\mathbf{x}^2 V_i, u_i, \hat u_{-i}   ), \quad \forall  i.
\end{align}

We will calculate $\hat u_i$ under the following conditions:  for all $(t, \mathbf{x}) \in [0, T] \times \mathbb{R}^{Nn}$,
\begin{align}
\label{Ri>0}
&   R +  \frac{1}{2} \mathbf{B}_i^T \frac{\partial^2 V_i(t, \mathbf{x})}{\partial \mathbf{x}^2} \mathbf{B}_i >0 ,
\\
&  I + \frac{1}{2} \sum_{k=1}^N \Big(  R +  \frac{1}{2} \mathbf{B}_k^T 
\frac{\partial^2 V_k(t, \mathbf{x})}{\partial \mathbf{x}^2} \mathbf{B}_k    \Big)^{-1} \mathbf{B}_0^T \frac{\partial^2 V_k(t, \mathbf{x})}{\partial \mathbf{x}^2} \mathbf{B}_0    \  \text{is invertible} ,
\label{inv}\\
& R +  \frac{1}{2} \mathbf{B}_i^T \frac{\partial^2 V_i(t, \mathbf{x})}{\partial \mathbf{x}^2} \mathbf{B}_i
+  \frac{1}{2} \mathbf{B}_0^T \frac{\partial^2 V_i(t, \mathbf{x})}{\partial \mathbf{x}^2} \mathbf{B}_0 >0 .
 \label{R0i>0}
\end{align}
By \eqref{Hminimizer}, we derive
\begin{align}
 0 = \widehat{\mathbf{B}}_i^T \frac{\partial V_i}{\partial \mathbf{x}}
 + \mathbf{B}_0^T \frac{\partial^2 V_i}{\partial \mathbf{x}^2} \mathbf{B}_0 \sum_{i\ne k=1}^N \hat u_k + \mathbf{B}_0^T \frac{\partial^2 V_i}{\partial \mathbf{x}^2} \mathbf{B}_0 \hat u_i + \mathbf{B}_i^T \frac{\partial^2 V_i}{\partial \mathbf{x}^2}
\mathbf{B}_i \hat u_i + 2 R \hat u_i , \label{NEuhat}
\end{align}
which implies that
\begin{align}
  \hat u_i
= -  \frac{1}{2} \Big(  R  +  \frac{1}{2}\mathbf{B}_i^T \frac{\partial^2 V_i}{\partial \mathbf{x}^2}\mathbf{B}_i   \Big)^{-1}
 \Big( \widehat{\mathbf{B}}_i^T \frac{\partial V_i}{\partial \mathbf{x}}
 + \mathbf{B}_0^T \frac{\partial^2 V_i}{\partial \mathbf{x}^2} \mathbf{B}_0 \sum_{k=1}^N \hat u_k \Big) . \label{uisumu}
\end{align}
Adding up the $N$ equations in \eqref{uisumu} leads to
\begin{align}
 \sum_{i=1}^N  \hat u_i
= -  \frac{1}{2} \sum_{i=1}^N \Big(  R  +  \frac{1}{2}\mathbf{B}_i^T \frac{\partial^2 V_i}{\partial \mathbf{x}^2}\mathbf{B}_i  \Big)^{-1}
 \Big( \widehat{\mathbf{B}}_i^T \frac{\partial V_i}{\partial \mathbf{x}}
 + \mathbf{B}_0^T \frac{\partial^2 V_i}{\partial \mathbf{x}^2} \mathbf{B}_0 \sum_{k=1}^N \hat u_k \Big) , \notag
\end{align}
which
under condition \eqref{inv} yields
\begin{align}
  \sum_{k = 1}^N \hat u_k
 = & - \frac{1}{2} \Big[ I + \frac{1}{2} \sum_{k=1}^N ( R + \frac{1}{2} \mathbf{B}_k^T \frac{\partial^2 V_k}{\partial \mathbf{x}^2} \mathbf{B}_k   )^{-1} \mathbf{B}_0^T \frac{\partial^2 V_k}{\partial \mathbf{x}^2} \mathbf{B}_0 \Big]^{-1} \cdot \notag \\
&
\quad \sum_{k=1}^N \Big(  R + \frac{1}{2} \mathbf{B}_k^T \frac{\partial^2 V_k}{\partial \mathbf{x}^2} \mathbf{B}_k     \Big)^{-1}  \widehat{\mathbf{B}}_k^T \frac{\partial V_k}{\partial \mathbf{x}}
\notag \\
 = : &\ \mathbf{M} .  \label{sumu}
\end{align}
Combining \eqref{uisumu} and \eqref{sumu} gives that
\begin{align}
  \hat u_i = -  \frac{1}{2} \Big(  R  +  \frac{1}{2}\mathbf{B}_i^T \frac{\partial^2 V_i}{\partial \mathbf{x}^2}\mathbf{B}_i   \Big)^{-1}
 \Big( \widehat{\mathbf{B}}_i^T \frac{\partial V_i}{\partial \mathbf{x}}
 + \mathbf{B}_0^T \frac{\partial^2 V_i}{\partial \mathbf{x}^2} \mathbf{B}_0 \mathbf{M}  \Big) .
\label{uieqm}
\end{align}

We substitute \eqref{uieqm} into the right hand side of
\eqref{HJBV} to obtain
\begin{align}
\label{VPDE}
 - \frac{\partial V_i}{\partial t} = & \frac{1}{4} \Big[ \mathbf{B}_0^T \frac{\partial^2 V_i}{\partial \mathbf{x}^2} \mathbf{B}_0 \mathbf{M}
 - \widehat{\mathbf{B}}_i^T \frac{\partial V_i}{\partial \mathbf{x}} \Big]^T   \Big(  R  +  \frac{1}{2}\mathbf{B}_i^T \frac{\partial^2 V_i}{\partial \mathbf{x}^2}\mathbf{B}_i   \Big)^{-1}
 \Big[ \mathbf{B}_0^T \frac{\partial^2 V_i}{\partial \mathbf{x}^2} \mathbf{B}_0 \mathbf{M} + \widehat{\mathbf{B}}_i^T \frac{\partial V_i}{\partial \mathbf{x}} \Big]   \\
  & - \frac{1}{2} \frac{\partial^T V_i}{\partial \mathbf{x}} \sum_{i\ne k = 1 }^N
 \widehat{\mathbf{B}}_k   \Big(  R  +  \frac{1}{2}\mathbf{B}_k^T \frac{\partial^2 V_k}{\partial \mathbf{x}^2}\mathbf{B}_k   \Big)^{-1}
 \Big[    \mathbf{B}_0^T \frac{\partial^2 V_k }{\partial \mathbf{x}^2} \mathbf{B}_0 \mathbf{M}  + \widehat{\mathbf{B}}_k^T \frac{\partial V_k }{\partial \mathbf{x}}  \Big]  \notag \\
& + \frac{1}{8} \sum_{i\ne k=1}^N
 \Big[  \mathbf{B}_0^T \frac{\partial^2 V_k }{\partial \mathbf{x}^2} \mathbf{B}_0 \mathbf{M}   + \widehat{\mathbf{B}}_k^T \frac{\partial V_k }{\partial \mathbf{x}}  \Big]^T
  \Big(  R  +  \frac{1}{2}\mathbf{B}_k^T \frac{\partial^2 V_k}{\partial \mathbf{x}^2}\mathbf{B}_k   \Big)^{-1}
   \cdot \notag \\
& \qquad\qquad \mathbf{B}_k^T \frac{\partial^2 V_i}{\partial \mathbf{x}^2} \mathbf{B}_k
 \Big(  R  +  \frac{1}{2}\mathbf{B}_k^T \frac{\partial^2 V_k}{\partial \mathbf{x}^2}\mathbf{B}_k   \Big)^{-1}
 \Big[  \mathbf{B}_0^T \frac{\partial^2 V_k }{\partial \mathbf{x}^2} \mathbf{B}_0 \mathbf{M}   + \widehat{\mathbf{B}}_k^T \frac{\partial V_k }{\partial \mathbf{x}}  \Big]  \notag \\
& + \frac{1}{2} \mathbf{M}^T \mathbf{B}_0^T \frac{\partial^2 V_i}{\partial \mathbf{x}^2} \mathbf{B}_0 \mathbf{M}
+ \mathbf{x}^T \mathbf{Q}_i \mathbf{x}  + \frac{\partial^T V_i}{\partial \mathbf{x}} \mathbf{A} \mathbf{x} ,
\notag \\
%%%
V_i( T, \mathbf{x} ) = & \mathbf{x}^T \mathbf{Q}_{i f} \mathbf{x}  , \quad 1\leq i \leq N  , \notag
\end{align}
subject to the conditions \eqref{Ri>0}, \eqref{inv} and \eqref{R0i>0}.

We are interested in a solution of the form
\begin{align}
V_i(t, \mathbf{x}) = \mathbf{x}^T \mathbf{P}_i(t) \mathbf{x} , \quad
 1\leq i \leq N ,
\label{Vansatz}
\end{align}
where  $\mathbf{P}_i(t)$ is a symmetric matrix function of $t\in [0, T]$ and is differentiable in $t$.
Substituting \eqref{Vansatz} into \eqref{VPDE}, we obtain the ODE system for $\mathbf{P}_i$, $1\leq i \leq N$:
\begin{align}
\begin{cases}
- \dot{\mathbf{P}}_i =  \mathbf{P}_i \mathbf{A}  + \mathbf{A}^T \mathbf{P}_i + \mathbf{Q}_i
 - \mathbf{P}_i \widehat{\mathbf{B}}_i
(R + \mathbf{B}_i^T \mathbf{P}_i \mathbf{B}_i )^{-1}
   \widehat{\mathbf{B}}_i^T  \mathbf{P}_i
 \\
%%%
\qquad\quad
  +  \mathbf{M}_{0}^T \mathbf{B}_0^T \mathbf{P}_i \mathbf{B}_0
(R + \mathbf{B}_i^T \mathbf{P}_i \mathbf{B}_i )^{-1}
   \mathbf{B}_0^T \mathbf{P}_i \mathbf{B}_0  \mathbf{M}_{0}  \\
%%%
\qquad\quad - \mathbf{P}_i \sum_{i\ne k=1}^N
 \widehat{\mathbf{B}}_k
 (  R + \mathbf{B}_k^T  \mathbf{P}_k \mathbf{B}_k  )^{-1} (  \mathbf{B}_0^T \mathbf{P}_k \mathbf{B}_0  \mathbf{M}_{0}  +  \widehat{\mathbf{B}}_k^T  \mathbf{P}_k )     \\
%%%%%
\qquad \quad - \sum_{i\ne k=1}^N
(  \mathbf{B}_0^T \mathbf{P}_k \mathbf{B}_0 \mathbf{M}_{0}
 +  \widehat{\mathbf{B}}_k^T  \mathbf{P}_k )^T
 ( R + \mathbf{B}_k^T  \mathbf{P}_k \mathbf{B}_k  )^{-1}
 \widehat{\mathbf{B}}_k^T  \mathbf{P}_i
   \\
%%%%%
 \qquad\quad +  \sum_{i\ne k=1}^N (  \mathbf{B}_0^T \mathbf{P}_k \mathbf{B}_0  \mathbf{M}_{0}  +  \widehat{\mathbf{B}}_k^T  \mathbf{P}_k )^T   ( R + \mathbf{B}_k^T  \mathbf{P}_k \mathbf{B}_k  )^{-1}
\mathbf{B}_k^T \mathbf{P}_i \mathbf{B}_k  \cdot  \\
%%%%%%%
 \qquad\quad \ \
(  R + \mathbf{B}_k^T  \mathbf{P}_k \mathbf{B}_k  )^{-1}
( \mathbf{B}_0^T \mathbf{P}_k \mathbf{B}_0  \mathbf{M}_{0}  +  \widehat{\mathbf{B}}_k^T  \mathbf{P}_k )
 +  \mathbf{M}_{0}^T \mathbf{B}_0^T \mathbf{P}_i \mathbf{B}_0 \mathbf{M}_{0}
  ,  \\
%%%%%
 \mathbf{P}_i(T) = \mathbf{Q}_{if}  ,
\end{cases}
\label{ODEP}
\end{align}
subject to
\begin{align}
\begin{cases}
({\rm i}) \quad \ \ R + \mathbf{B}_i^T  \mathbf{P}_i(t) \mathbf{B}_i  >0  ,
\forall t \in [0, T] ,     \\
({\rm ii}) \quad \ R + \mathbf{B}_i^T \mathbf{P}_i(t) \mathbf{B}_i + \mathbf{B}_0^T \mathbf{P}_i(t) \mathbf{B}_0 >0 , \ \forall t\in[0, T] , \\
%%%
({\rm iii}) \quad I + \sum_{k=1}^N \left( R +   \mathbf{B}_k^T \mathbf{P}_k(t)
\mathbf{B}_k  \right)^{-1} \mathbf{B}_0^T  \mathbf{P}_k(t) \mathbf{B}_0  \ \text{is invertible} , \ \forall t \in [0, T] ,
\end{cases} \label{Pcon}
\end{align}
where
\begin{align}
\mathbf{M}_{0} := - \Big[ I + \sum_{k=1}^N \left( R +   \mathbf{B}_k^T \mathbf{P}_k \mathbf{B}_k
   \right)^{-1} \mathbf{B}_0^T  \mathbf{P}_k \mathbf{B}_0 \Big]^{-1}
 \sum_{k=1}^N \left(  R + \mathbf{B}_k^T  \mathbf{P}_k \mathbf{B}_k    \right)^{-1}
 \widehat{\mathbf{B}}_k^T  \mathbf{P}_k . \notag
\end{align}

In further analysis, if we just say $(\mathbf{P}_1, \cdots, \mathbf{P}_N )$ is a solution of \eqref{ODEP}, that means \eqref{Pcon} is in effect unless otherwise indicated. Condition \eqref{Pcon}--(ii) is not used in the vector field of the Riccati equation, but will play a role in the best response control problem later.

\begin{remark}
\label{rmk:uniqueP}
If the ODE system \eqref{ODEP} admits a solution $(\mathbf{P}_1, \cdots, \mathbf{P}_N )$ on $[0, T]$, then it is the unique solution since the vector field of the ODE system has a local Lipschitz property along the solution trajectory satisfying  \eqref{Pcon}--\emph{(i)} and \emph{(iii)}.
\end{remark}

The following theorem gives a sufficient condition for the existence of feedback Nash  strategies in terms of the Riccati equations \eqref{ODEP}. These strategies are called centralized due to the use of full state information by each player.
\begin{theorem}
\label{thm: FBNE}
If \eqref{ODEP} has a solution $(\mathbf{P}_1, \cdots, \mathbf{P}_N )$ on $[0, T]$, then the Nash game \eqref{X}--\eqref{Ji2} has a set of feedback Nash strategies $(\hat u_1, \cdots, \hat u_N)$ given by
\begin{align}
\label{uiP}
 \hat u_i(t) = - [  R + \mathbf{B}_i^T  \mathbf{P}_i(t)
 \mathbf{B}_i  ]^{-1}
 [ \mathbf{B}_0^T \mathbf{P}_i(t) \mathbf{B}_0  \mathbf{M}_{0}(t)   +  \widehat{\mathbf{B}}_i^T  \mathbf{P}_i(t) ] X(t) ,\  1\leq i \leq N .
\end{align}
\end{theorem}

\begin{proof}
See Appendix~\ref{appendix:pflmPsubmat}.
\end{proof}

The best response control problem in the proof of Theorem \ref{thm: FBNE} amounts to LQ optimal control with indefinite weights in the cost. The HJB equation \eqref{HJBV} is only used for constructing  \eqref{ODEP}. The rigorous proof of $\hat u_1$ as a best response strategy on $[t, T]$ given $(t, \mathbf{x}, \hat u_{-1})$ has been solely based  on  the Riccati equation system \eqref{ODEP} itself.

%%%%%%%%%%%%%%%%%
\section{Asymptotic solvability}
\label{sec:AS}

We start with a representation of the matrix $\mathbf{P}_i$ if the ODE system \eqref{ODEP} has a solution.
Write the $Nn \times Nn$ identity matrix $I_{Nn}$ as
$I_{Nn} = \diag[I_n, I_n, \cdots, I_n]$.
Let $J_{ij}$ denote the matrix obtained by exchanging the $i$th and $j$th rows of  submatrices in $I_{Nn}$.
\begin{lemma}
\label{lm:Psubmat}
Suppose \eqref{ODEP} has a solution $(\mathbf{P}_1, \cdots,
\mathbf{P}_N)$ on $[0, T]$.
Then $\mathbf{P}_i$, $1\leq i \leq N$, have the representation
\begin{align}
\label{Psubmat}
\mathbf{P}_1 = \begin{bmatrix}
\Pi_1^{N} & \Pi_2^{N} & \Pi_2^{N} & \cdots & \Pi_2^{N} \\
 \Pi_2^{NT} & \Pi_3^N & \Pi_4^N & \cdots & \Pi_4^N \\
 \Pi_2^{NT} & \Pi_4^N & \Pi_3^N & \cdots & \Pi_4^N \\
\vdots & \vdots & \vdots & \ddots & \vdots \\
 \Pi_2^{NT} & \Pi_4^N & \Pi_4^N & \cdots & \Pi_3^N \\
\end{bmatrix} , \quad
\mathbf{P}_i = J_{1i}^T \mathbf{P}_1 J_{1i} , \quad \forall
 2\leq i \leq N,
\end{align}
where $\Pi_1^N(t)$, $\Pi_3^N(t)$,  $\Pi_4^N(t) \in \mathcal{S}^{n}$,
and $\Pi_2^N(t)\in \mathbb{R}^{n\times n}$.
\end{lemma}
\begin{proof}
See Appendix~\ref{appendix:pflmPsubmat}.
\end{proof}

Following the route in \cite{HZ2020}, we introduce the notion of asymptotic solvability of the Nash game \eqref{Xi}--\eqref{Ji} (with $D = D_0 = 0$).
%%%%%%
\begin{definition}
\label{def:AS}
The Nash game \eqref{Xi}--\eqref{Ji} is asymptotically solvable if there exist $N_0>0$ and $c_0>0$ such that  the ODE system \eqref{ODEP}--\eqref{Pcon} has a solution $(\mathbf{P}_1, \cdots, \mathbf{P}_N)$ on $[0, T]$ for all $N\geq N_0$, and that
\begin{align}
& \sup_{N\geq N_0} \sup_{0\leq t \leq T}  ( |\Pi_1^N| +  N|\Pi_2^N| +  N^2|\Pi_3^N|
+  N^2|\Pi_4^N| ) < \infty , \label{AS2} \\
&   R + \mathbf{B}_i^T   \mathbf{P}_i \mathbf{B}_i  \ge c_0 I ,  \
  \ \forall t\in[0, T] ,\ \forall N\geq N_0,\label{AS3} \\
 &  I + \sum_{k=1}^N \left( R +   \mathbf{B}_k^T \mathbf{P}_k \mathbf{B}_k    \right)^{-1} \mathbf{B}_0^T  \mathbf{P}_k \mathbf{B}_0  \ \text{is invertible},\   \forall t \in [0, T] , \ \forall N \geq N_0 .
\label{AS4}
\end{align}
\end{definition}

\begin{remark}
The conditions \eqref{AS2}--\eqref{AS3} imply that
\begin{align}
  R + \mathbf{B}_i^T \mathbf{P}_i(t) \mathbf{B}_i + \mathbf{B}_0^T \mathbf{P}_i(t) \mathbf{B}_0 \ge (c_0/2) I, \quad \forall t,\notag
\end{align}
as long as a sufficiently large $N_0$ is chosen.
\end{remark}

Define the mapping
$\mathcal{R}_1: \mathcal{S}^n \to \mathcal{S}^n$ by
\begin{align}
 \mathcal{R}_1(Z) = R + B_1^T Z B_1 , \quad \mbox{for} \ Z \in \mathcal{S}^n . \notag
\end{align}
For  $\Lambda_k\in \mathcal{S}^n$, $k=1,3,4$, and $\Lambda_2\in\mathbb{R}^{n\times n} $,  we define the mappings:
\begin{align}
& \Psi_1(\Lambda_1) = \Lambda_1 B H B^T \Lambda_1 - \Lambda_1 A - A^T \Lambda_1 - Q  ,  \notag \\
%%%
& \Psi_2(\Lambda_1, \Lambda_2) =
 - \Lambda_1 G - \Lambda_2 (A+G) - A^T \Lambda_2
+ \Lambda_1 B H B^T \Lambda_2  + \Lambda_2 B H B^T \Lambda_1  \notag \\
 & \qquad\qquad\qquad
 + \Lambda_2 B H B^T \Lambda_2
  + Q \Gamma , \notag \\
%%%
& \Psi_3 (\Lambda_1, \Lambda_2, \Lambda_3, \Lambda_4)
 =
 - \Lambda_3 A  - A^T \Lambda_3- ( \Lambda_2^{T} + \Lambda_4 ) G
 - G^T(\Lambda_2 + \Lambda_4 )    \notag \\
& \hspace{3.3cm} - (\Lambda_1 + \Lambda_2^T ) B H B_0^T(\Lambda_1 + \Lambda_2 + \Lambda_2^T + \Lambda_4 ) B_0 H B^T (\Lambda_1 + \Lambda_2 )  \notag  \\
&\hspace{3.3cm}  +  \Lambda_3 B H B^T \Lambda_1
  + \Lambda_1 B H B^T \Lambda_3
 + \Lambda_4  B H B^T \Lambda_2 \notag   \\
& \hspace{3.3cm}
  + \Lambda_2^T B H B^T ( \Lambda_2 + \Lambda_4)
- \Lambda_1 B H B_1^T \Lambda_3 B_1 H B^T \Lambda_1
 - \Gamma^T Q \Gamma  , \notag \\
%%%
& \Psi_4 (\Lambda_1, \Lambda_2,  \Lambda_4)
 =- \Lambda_4 A - A^T \Lambda_4- (  \Lambda_2^T + \Lambda_4 ) G
  - G^T (   \Lambda_2 +  \Lambda_4 )
 - \Gamma^T Q \Gamma  \notag \\
& \hspace{2.7cm}
- (\Lambda_1 + \Lambda_2^T ) B H B_0^T
 ( \Lambda_1 + \Lambda_2 + \Lambda_2^T + \Lambda_4 ) B_0 H B^T ( \Lambda_1 + \Lambda_2 ) \notag \\
 & \hspace{2.7cm}
   + \Lambda_4 B H B^T ( \Lambda_1   + \Lambda_2 )
  + (\Lambda_1 + \Lambda_2^T ) B H B^T \Lambda_4
 + \Lambda_2^T B H B^T  \Lambda_2  , \notag
\end{align}
where we denote $H = ( \mathcal{R}_1(\Lambda_1) )^{-1}$ provided that the inverse matrix exists. It is clear that $\Psi_k$, $k=1,3,4$, are
$\mathcal{S}^n$-valued.

We introduce the following ODE system
\begin{align}
& \begin{cases}
   \dot \Lambda_1 =  \Psi_1(\Lambda_1) ,    \\
 \Lambda_1 (T) =  Q_f , \quad
 \mathcal{R}_1(\Lambda_1(t) ) >0, \quad \forall  t\in[0, T]  ,
\end{cases} \label{ODELam1} \\
%%%%%%%
&  \dot\Lambda_2 = \Psi_2(\Lambda_1, \Lambda_2) ,    \qquad
\Lambda_2 (T) =   - Q_f \Gamma_f  , \label{ODELam2} \\
%%%%%%%%%%%%%
& \dot \Lambda_3 = \Psi_3(\Lambda_1, \Lambda_2, \Lambda_3, \Lambda_4)  ,  \qquad
\Lambda_3 (T) =  \Gamma_f^T Q_f \Gamma_f ,
\label{ODELam3} \\
%%%%%%%%%%%
&   \dot \Lambda_4 =
 \Psi_4(\Lambda_1, \Lambda_2,  \Lambda_4 )  ,   \qquad
\Lambda_4 (T) = \Gamma_f^T Q_f \Gamma_f  . \label{ODELam4}
\end{align}

\begin{remark}
\label{rmk:solLam}
Note that \eqref{ODELam1} is the Riccati equation associated with an optimal control problem with controlled diffusion.
If \eqref{ODELam1}--\eqref{ODELam2} admits a solution $(\Lambda_1, \Lambda_2)$, substituting $(\Lambda_1, \Lambda_2)$ into
\eqref{ODELam3}--\eqref{ODELam4} gives a first order linear ODE system of
$(\Lambda_3, \Lambda_4)$, which then admits a unique solution on $[0, T]$.
\end{remark}

\begin{remark}
If $B_1 = 0$ and \eqref{ODELam1}--\eqref{ODELam2} has a solution on $[0,T]$,  from \eqref{ODELam3}--\eqref{ODELam4} we obtain a first order linear homogeneous ODE of $\Lambda_3 - \Lambda_4$ with zero terminal condition $\Lambda_3(T) - \Lambda_4(T) = 0$, which implies  that $\Lambda_3 - \Lambda_4 = 0$ on $[0, T]$. Such a representation by three submatrices is similar to \cite[Theorem 3]{HZ2020}.
\end{remark}

The following theorem characterizes asymptotic solvability of the Nash game \eqref{Xi}--\eqref{Ji} in terms of the low-dimensional ODE system \eqref{ODELam1}--\eqref{ODELam2}. The proof is postponed near the end of this section.
\begin{theorem}
\label{thm:NSAS}
The Nash game \eqref{Xi}--\eqref{Ji} has asymptotic solvability
if and only if \eqref{ODELam1}--\eqref{ODELam2} has a solution $(\Lambda_1, \Lambda_2)$ on $[0, T]$.
\end{theorem}

Following the rescaling method in \cite{HY2020b,HZ2020,MH2020},
we define
\begin{align}
 \Lambda_1^N(t)=  \Pi_1^N(t),  \
\Lambda_2^N(t) = N\Pi_2^N(t)  ,  \
 \Lambda_3^N(t) = N^2 \Pi_3^N(t) , \
\Lambda_4^N(t)=N^2\Pi_4^N(t)  . \label{PLam}
\end{align}

We introduce the following ODE system for $(\Lambda_1^N, \cdots, \Lambda_4^N)$:
%%%%%%%%%%
\begin{align}
& \begin{cases}
   \dot \Lambda_1^N =  \Psi_1(\Lambda_1^N) + g_1^N ,    \\
 \Lambda_1^{N} (T) =   (I - \Gamma_f^T/N ) Q_f  (I - \Gamma_f/N ) , \quad
  \mathcal{R}_1(\Lambda_1^N(t))  >0, \quad \forall  t\in[0, T]  ,
\end{cases} \label{ODELam1N} \\
%%%%%%%
& \begin{cases}
  \dot\Lambda_2^N = \Psi_2(\Lambda_1^N, \Lambda_2^N)
 + g_2^N  ,    \\
%%%
\Lambda^N_2 (T) =  - (I - \Gamma_f^T/N ) Q_f \Gamma_f     ,
\end{cases} \label{ODELam2N} \\
%%%%%%%%%%%%%
&\begin{cases}
  \dot \Lambda^N_3 = \Psi_3(\Lambda_1^N, \Lambda_2^N, \Lambda_3^N, \Lambda_4^N) + g_3^N  ,  \\
%%%
\Lambda_3^N (T) =  \Gamma_f^T Q_f \Gamma_f ,
\end{cases}\label{ODELam3N} \\
%%%%%%%%%%%
& \begin{cases}
  \dot \Lambda_4^N =
 \Psi_4(\Lambda_1^N, \Lambda_2^N,  \Lambda_4^N )
 + g_4^N ,    \\
%%%%%%
\Lambda_4^N (T) = \Gamma_f^T Q_f \Gamma_f  ,
\end{cases} \label{ODELam4N}
\end{align}
where  $g_k^N$, $1\leq k \leq 4$, are perturbation terms. We have
\begin{align}
  g_1^N =&  - [ \Lambda_2^N B K^N B^T S_{12}^N
 +  S_{12}^{NT} B K^{NT} B^T  \Lambda_2^{NT}
] (N-1)/N^3   \notag \\
 &
 + [ \Lambda_2^N B H^N B^T \Lambda_2^N  + \Lambda_2^{NT} B H^N B^T \Lambda_2^{NT} ] (N-1)/N^2
 \notag \\
 & - [ S_{12}^{NT} B K^{NT}/N - \Lambda_2^{NT} B H^N ] B_1^T \Lambda_3^N B_1 [ K^N B^T S_{12}^N/N - H^N B^T \Lambda_2^N ]
(N-1)/N^4  \notag \\
& -  [ \Lambda_1^N G + G^T \Lambda_1^N ]/N
- [ \Lambda_2^N G  + G^T  \Lambda_2^{NT} ] (N-1)/N^2
 \notag \\
 & - S_{12}^{NT} B F^N B^T  S_{12}^N /N^2   - ( \Gamma^T Q \Gamma/N - \Gamma^T Q  - Q \Gamma )/N  \notag
\end{align}
and
\begin{align}
& H^N = (R+ B_1^T \Lambda_1^N B_1 )^{-1} ,  \notag \\
&S^N =   \Lambda_1^N + ( \Lambda_2^N + \Lambda_2^{NT} )(N-1)/N
 + [ \Lambda_3^N
  + \Lambda_4^N(N-2)] (N-1)/N^2  , \notag \\
&S_{12}^N =  \Lambda_1^N + \Lambda_2^N (N-1)/N ,  \notag \\
&S_{34}^N =   \Lambda_3^N/N^2 + \Lambda_4^N (N-2)/N^2 , \notag\\
 & K^N = H^N B_0^T S^N B_0 (I +  H^N B_0^T S^N B_0/N )^{-1} H^N , \notag \\
& F^N =  H^N (I +   B_0^T S^N B_0 H^N /N )^{-1}
( B_0^T S^N B_0 + B_0^T S^N B_0 H^N B_0^T S^N B_0 / N^2  )
\cdot \notag \\
&\qquad\quad   (I +  H^N B_0^T S^N B_0/N )^{-1} H^N   . \notag
\end{align}
The other terms $g_k^N$, $k=2,3,4$, are listed in Appendix \ref{appendix:gdN}. They depend on $S_{34}^N$ above.
The mappings $g_k^N$, $1\le k\le 4$, are defined for
$\Lambda^N_k\in \mathcal{S}^n$, $k=1,3,4$, and $\Lambda^N_2\in\mathbb{R}^{n\times n} $. If  \eqref{ODELam1N}--\eqref{ODELam4N} has a solution on $[0,T]$, then $\Lambda^N_k(t)$ is $\mathcal{S}^n$-valued for $k=1, 3, 4$.
The ODE system \eqref{ODELam1N}--\eqref{ODELam4N} is essentially derived from \eqref{ODEP} by use of the new variables \eqref{PLam}. However, \eqref{ODELam1N}--\eqref{ODELam4N} can stand alone without being immediately related to \eqref{ODEP}. If $(\Lambda_1^N, \cdots, \Lambda_4^N)$ is a solution, the inverse
$ (I +  H^N B_0^T S^N B_0/N )^{-1}$ necessarily exists for all $t\in[0, T]$; such a solution is unique.

For
$\Lambda^N_k\in \mathcal{S}^n$, $k=1,3,4$, and $\Lambda^N_2\in\mathbb{R}^{n\times n} $, define the mappings
\begin{align}
\xi(\Lambda_1^N, \Lambda_2^N, \Lambda_3^N, \Lambda^N_4)=& \mathcal{R}_1 ( \Lambda_1^N ) + B_0^T  [\Lambda_1^N + ( \Lambda_2^N + \Lambda_2^{NT} ) (N-1)/N \label{zetaL14}  \\
&\hskip 0cm + \Lambda_3^N(N-1)/N^2
 + \Lambda_4^N (N-1)(N-2)/N^2 ] B_0 /N^2, \nonumber\\
\xi_0(\Lambda_1^N, \Lambda_2^N, \Lambda_3^N, \Lambda^N_4)= & I +  \left( \mathcal{R}_1 ( \Lambda_1^N )  \right)^{-1}
 B_0^T  [\Lambda_1^N + ( \Lambda_2^N + \Lambda_2^{NT} ) (N-1)/N \label{xi0Lam14} \\
&+ \Lambda_3^N(N-1)/N^2 + \Lambda_4^N (N-1)(N-2)/N^2 ] B_0 /N.  \notag
\end{align}
It is easy to show that
\begin{align}\label{IHxi0}
I +  H^N B_0^T S^N B_0/N =  \xi_0(\Lambda_1^N, \Lambda_2^N, \Lambda_3^N, \Lambda^N_4) .
\end{align}

\begin{lemma}
\label{lm:solPLam}
\emph{(i)} Suppose \eqref{ODEP}--\eqref{Pcon} has a solution $(\mathbf{P}_1, \cdots, \mathbf{P}_N)$ on $[0, T]$, and let
$(\Lambda_1^N, \Lambda_2^N, \Lambda_3^N, \Lambda^N_4)$ be defined using \eqref{Psubmat} and \eqref{PLam}. Then $(\Lambda_1^N, \Lambda_2^N,\Lambda_3^N, \Lambda^N_4 )$ satisfies
\eqref{ODELam1N}--\eqref{ODELam4N}.

\emph{(ii)} Conversely, if \eqref{ODELam1N}--\eqref{ODELam4N} admits a solution $(\Lambda_1^N, \Lambda_2^N,\Lambda_3^N, \Lambda^N_4  )$ on $[0, T]$, and such a solution further satisfies
 \begin{align}
\xi(\Lambda_1^N(t), \Lambda_2^N(t), \Lambda_3^N(t), \Lambda^N_4(t)) >0 \nonumber
\end{align}
for all $t\in [0,T]$, then   \eqref{ODEP}--\eqref{Pcon} has a solution $(\mathbf{P}_1, \cdots, \mathbf{P}_N)$ on $[0, T]$. Moreover, $\mathbf{P}_i$ may be determined in terms of  the above $(\Lambda_1^N, \Lambda_2^N,\Lambda_3^N, \Lambda^N_4)$ using \eqref{Psubmat} .
\end{lemma}

\begin{proof}
(i) By \eqref{Psubmat} and \eqref{PLam},  we have
\begin{align}
R + \mathbf{B}_i^T  \mathbf{P}_i \mathbf{B}_i
= \mathcal{R}_1 (\Lambda_1^N ) .
\label{R1(P)=R1(Lam)}
\end{align}
By condition \eqref{Pcon}--(i), $R + \mathbf{B}_i^T  \mathbf{P}_i \mathbf{B}_i>0$. Therefore, $ \mathcal{R}_1( \Lambda_1^N(t) )>0$ on $[0, T]$.
It can be shown that
\begin{align}
\label{invP=invLam}
 & I + \sum_{k=1}^N \left( R +   \mathbf{B}_k^T \mathbf{P}_k(t)
\mathbf{B}_k  \right)^{-1} \mathbf{B}_0^T  \mathbf{P}_k(t) \mathbf{B}_0
  =\xi_0(\Lambda_1^N(t), \Lambda_2^N(t), \Lambda_3^N(t), \Lambda^N_4(t)).
\end{align}
We substitute \eqref{Psubmat}  into \eqref{ODEP} and change to the variables $\Lambda_k^N$, $1\le k\le 4$,  to verify the equalities \eqref{ODELam1N}--\eqref{ODELam4N}, for which the inverse   $(I +  H^N B_0^T S^N B_0/N )^{-1}$ exists by condition \eqref{Pcon}--(iii), \eqref{IHxi0} and \eqref{invP=invLam}.

(ii) If \eqref{ODELam1N}--\eqref{ODELam4N} admits a solution $(\Lambda_1^N, \Lambda_2^N,\Lambda_3^N, \Lambda^N_4  )$ on $[0, T]$, let $\mathbf{P}_i$ be defined by \eqref{Psubmat} and \eqref{PLam}.
By $\mathcal{R}_1(\Lambda_1^N)>0$ in \eqref{ODELam1N}, we have
$R+\mathbf{B}_i^T \mathbf{P}_i \mathbf{B}_i >0$.
We can verify
\begin{align}
 & R + \mathbf{B}_i^T \mathbf{P}_i \mathbf{B}_i + \mathbf{B}_0^T \mathbf{P}_i  \mathbf{B}_0= \xi(\Lambda_1^N, \Lambda_2^N, \Lambda_3^N, \Lambda^N_4)   \label{R2(P)=R2(Lam)}
\end{align}
so that  $R + \mathbf{B}_i^T \mathbf{P}_i \mathbf{B}_i + \mathbf{B}_0^T \mathbf{P}_i  \mathbf{B}_0 >0$ for all $t\in [0,T]$.
Note that
$(I +  H^N B_0^T S^N B_0/N )^{-1}$ in \eqref{ODELam1N}--\eqref{ODELam4N}
exists for all $t\in [0,T]$.  Recalling \eqref{xi0Lam14}--\eqref{IHxi0} and \eqref{invP=invLam}, we see that
\begin{align}
  I + \sum_{k=1}^N \left( R +   \mathbf{B}_k^T \mathbf{P}_k(t)
\mathbf{B}_k  \right)^{-1} \mathbf{B}_0^T  \mathbf{P}_k(t) \mathbf{B}_0 \notag
\end{align}
is invertible for all $t\in [0,T]$.
Now it is straightforward to verify that $(\mathbf{P}_1, \cdots, \mathbf{P}_N)$ defined above solves  \eqref{ODEP} subject to \eqref{Pcon}.
\end{proof}

%%%%%%%
\begin{proof}[Proof of Theorem~\ref{thm:NSAS}]

(i) -- Necessity. Suppose the game \eqref{Xi}--\eqref{Ji} has asymptotic solvability, where $N_0$ and $c_0>0$ have been selected in \eqref{AS2}--\eqref{AS4}.
By Lemma \ref{lm:solPLam}--(i), for all $N\ge N_0$, \eqref{ODELam1N}--\eqref{ODELam4N} has a solution $(\Lambda_1^N, \Lambda_2^N,\Lambda_3^N, \Lambda^N_4  )$  on $[0,T]$, and by \eqref{AS2}--\eqref{AS3} and  \eqref{PLam}, we have
\begin{align}
& \sup_{N\geq N_0} \sup_{0\leq t \leq T}  ( |\Lambda_1^N| +  |\Lambda_2^N| +  |\Lambda_3^N|
+  |\Lambda_4^N| ) < \infty , \label{LamB2} \\
&  \mathcal{R}_1 ( \Lambda_1^N(t) )\ge c_0 I , \quad \forall  t\in [0, T] , \quad \forall  N\geq N_0 .
\label{LamB3}
\end{align}
 We write \eqref{ODELam1N} in the integral form
\begin{align}
& \Lambda_1^N(t) = \Lambda_1^N(T) - \int_t^T [\Psi_1(\Lambda_1^N) + g_1^N] d \tau ,  \notag  %\label{intLam1N}
\end{align}
and do the same for $\Lambda_2^N$, $\Lambda_3^N$ and $\Lambda_4^N$. By
\eqref{LamB2}--\eqref{LamB3} we obtain $\sup_{0\le t\le T, k\le 4}|g_k^N| = O(1/N)$.
Then the functions $\{(\Lambda_1^N(\cdot), \Lambda_2^N(\cdot)),\Lambda_3^N(\cdot)), \Lambda_4^N(\cdot)) \}_{N\geq N_0}$ are uniformly bounded and equicontinuous on $[0, T]$.
By Arzel\`{a}--Ascoli theorem~\cite{Y1980}, there exists a subsequence
$\{ ( \Lambda_1^{N_j}(\cdot) , \Lambda_2^{N_j}(\cdot),\Lambda_3^{N_j}(\cdot) ,\Lambda_4^{N_j}(\cdot) ) \}_{j\geq 1}$  that  converges to
$(\Lambda_1^\ast, \Lambda_2^\ast, \Lambda_3^\ast , \Lambda_4^\ast )$ uniformly on $[0, T]$ as $j\to\infty$.
It is easy to see that $(\Lambda_1^\ast, \Lambda_2^\ast, \Lambda_3^\ast, \Lambda_4^\ast )$ solves the system \eqref{ODELam1}--\eqref{ODELam4} and $\mathcal{R}_1 ( \Lambda_1^\ast(t) ) \geq c_0 I$ for all $t\in [0,T]$.

(ii) --
Sufficiency.
Suppose \eqref{ODELam1}--\eqref{ODELam2} has a solution so that we can obtain $(\Lambda_1, \Lambda_2, \Lambda_3, \Lambda_4)$ from
\eqref{ODELam1}--\eqref{ODELam4}.
We proceed to check the solution of \eqref{ODELam1N}--\eqref{ODELam4N}, which now stands alone without using \eqref{ODEP}.
Following the method  in the sufficiency proof of Theorem 3.1 in \cite{HY2020b}, we specify a thin ``tube", surrounding the solution trajectory $(\Lambda_1, \Lambda_2, \Lambda_3, \Lambda_4)$, $t\in [0,T]$, of this form:
\begin{align}
\label{tubeC}
{\cal C}=& \{(t, Z_1, Z_2, Z_3 ,Z_4)\in [0,T]\times \mathcal{S}^n\times
\mathbb{R}^{n\times n}\times \mathcal{S}^n\times \mathcal{S}^n:\\
&\qquad  \mbox{$\sum_{k\le 4}$}|Z_k-\Lambda_k(t)|<\delta_0  \}, \nonumber
\end{align}
where $\delta_0>0$ is a sufficiently small but fixed constant, and next show that
for all sufficiently large $N$, the solution of \eqref{ODELam1N}--\eqref{ODELam4N} starting from the terminal condition will always remain in this tube. This establishes the global existence of  solutions on $[0,T]$,
and the detailed steps are exactly the same as in \cite{HY2020b}.  Specifically,
   it can be shown that there exist $\hat N_0$ and $c_0>0$ such that  we have  the following: (a) \eqref{ODELam1N}--\eqref{ODELam4N}
has a solution $(\Lambda_1^N, \Lambda_2^N,\Lambda_3^N, \Lambda_4^N) $ remaining in the tube \eqref{tubeC} on $[0,T]$ for all $N\ge \hat N_0$; (b)
\begin{align}
& \sup_{N\geq \hat N_0} \sup_{0\leq t \leq T}  ( |\Lambda_1^N| +  |\Lambda_2^N| +  |\Lambda_3^N|
+  |\Lambda_4^N| ) < \infty , \label{LamB2s} \\
&   \mathcal{R}_1 ( \Lambda_1^N(t) ) \ge c_0 I, \quad \forall t\in [0, T] , \quad \forall  N\geq \hat N_0 ;
\label{LamB3s}
\end{align}
(c) for $\xi_0(\cdot)$ defined in \eqref{xi0Lam14},  $\xi_0(\Lambda_1^N(t), \Lambda_2^N(t), \Lambda_3^N(t), \Lambda^N_4(t)) $ is invertible for all $ N\ge \hat N_0$,
so that the term $(I +  H^N B_0^T S^N B_0/N )^{-1}$ in  \eqref{ODELam1N}--\eqref{ODELam4N}
is well defined.

If $\hat N_1>\hat N_0$ is sufficiently large, by \eqref{LamB3s} we can ensure that
\begin{align}\xi(\Lambda_1^N(t), \Lambda_2^N(t), \Lambda_3^N(t), \Lambda^N_4(t))
 >(c_0/2) I, \quad \forall t\in [0,T], \ \forall N\ge \hat N_1, \label{R1c02I}
\end{align}
where  $\xi(\cdot)$ is defined  in \eqref{zetaL14}.
By \eqref{R1c02I}, we apply Lemma \ref{lm:solPLam}--(ii) to obtain $(\mathbf{P}_1, \cdots, \mathbf{P}_N)$ for \eqref{ODEP} whenever $N\ge \hat N_1$. By \eqref{R1c02I} and \eqref{invP=invLam}, we see that \eqref{AS4} holds for all $N\ge \hat N_1$.
Subsequently, asymptotic solvability holds.
\end{proof}

\begin{corollary}
\label{cor:NSAS:LamN-Lam}
If \eqref{ODELam1}--\eqref{ODELam2} has a solution $(\Lambda_1, \Lambda_2)$ on $[0, T]$, then there exists $\hat N_0>0$ such that for each $N\geq \hat N_0$, \eqref{ODELam1N}--\eqref{ODELam4N} has a solution
$(\Lambda_1^N, \Lambda_2^N, \Lambda_3^N, \Lambda_4^N)$ on
$[0, T]$ and moreover,
  %\begin{align}
$ \sup_{t\in[0, T], k\le 4}
  |\Lambda_k^N(t) - \Lambda_k(t)| =O(1/N). $ %\notag
 %\label{LamN-Lam}
 %\end{align}
\end{corollary}
\begin{proof}
 Since \eqref{ODELam1}--\eqref{ODELam4} has a solution on $[0, T]$, we take a sufficiently thin tube as in \eqref{tubeC}.
Then by the sufficiency proof of  Theorem~\ref{thm:NSAS}, there exists $\hat N_0>0$ such that for each $N\geq \hat N_0$, \eqref{ODELam1N}--\eqref{ODELam4N} has a solution
$(\Lambda_1^N, \Lambda_2^N, \Lambda_3^N, \Lambda_4^N)$ on
$[0, T]$, which is always within the tube.
The desired result then follows from Gr\"{o}nwall's lemma. See similar estimates  in \cite[Corollary 3.1]{HY2020b}.
\end{proof}

%%%%%%%%%%%%%%%%%%%%%%%%%%%%%%
\section{Decentralized strategies}

\label{sec:DecentralStrats}

By Theorem~\ref{thm:NSAS}, the Nash  game \eqref{Xi}--\eqref{Ji} has asymptotic solvability if and only if \eqref{ODELam1}--\eqref{ODELam2} admits a solution $(\Lambda_1, \Lambda_2)$ on $[0, T]$.
We introduce the following assumptions:
\begin{assumption}
\label{assm:solLam}
The ODE system~\eqref{ODELam1}--\eqref{ODELam2} has a solution $(\Lambda_1, \Lambda_2)$ on $[0, T]$.
\end{assumption}

For $X_i(0)$,  denote the covariance matrix
$\Sigma_0^i= \mathbb{E}\{ [X_i(0)-\mathbb{E}X_i(0)] [X_i(0)-\mathbb{E}X_i(0)]^T\}$.

\begin{assumption}
\label{assm:initialX}
The initial states $\{X_i(0), i\geq 0\}$ are independent. There exist  $\mu_0\in\mathbb{R}^n$ and a constant $C_\Sigma$, both independent of $N$, such that $\mathbb{E}X_i(0) = \mu_0$ and
$|\Sigma_0^i| \leq C_\Sigma$ for all $ i$.
\end{assumption}

Under Assumption \ref{assm:solLam}, the sufficiency proof of Theorem \ref{thm:NSAS} shows that there exists $\hat N_1$ such that
\eqref{ODELam1N}--\eqref{ODELam4N} has a solution $(\Lambda_1^N, \Lambda_2^N, \Lambda_3^N, \Lambda_4^N)$  for all $N\ge \hat N_1$. By Lemma~\ref{lm:solPLam}--(ii), we determine $\mathbf{P}$ in~\eqref{ODEP} by using \eqref{Psubmat} and~\eqref{PLam}, and obtain the Nash equilibrium strategies
\eqref{uiP}, which are displayed below:
\begin{align}
  \hat u_i(t) = - [  R + B_1^T  \Lambda_1^N(t) B_1  ]^{-1}
 [ \mathbf{B}_0^T \mathbf{P}_i(t) \mathbf{B}_0  \mathbf{M}_{0}(t)   +  \widehat{\mathbf{B}}_i^T  \mathbf{P}_i(t) ] X(t) ,\  1\leq i \leq N .
\notag
\end{align}
Throughout this section  we assume $N\ge \hat N_1$.
Before further analysis we introduce some notation:
\begin{align}
& \Theta(t) = ( \mathcal{R}_1(\Lambda_1(t)) )^{-1} B^T \Lambda_1(t) ,
\quad \Theta_1(t) = ( \mathcal{R}_1(\Lambda_1(t)) )^{-1} B^T \Lambda_2(t) , \notag \\
&  \widehat{\Theta}(t)  = I_N \otimes  \Theta(t) ,
\quad \widehat{\Theta}_1 = \mathbf{1}_{N\times 1} \otimes \Theta_1 , \notag \\
%%%%%%%%
& \mathbf{e}_i = ( e_i^N  \otimes I_n )^T = (0, \cdots, 0, I_n, 0, \cdots, 0 ) \in \mathbb{R}^{n \times Nn} , \notag \\
& \widehat{\mathbf{B}} = (\widehat{\mathbf{B}}_1, \cdots, \widehat{\mathbf{B}}_N ) \in \mathbb{R}^{Nn\times N n_1}, \quad
\mathbf{I} = (I_n, \cdots, I_n) \in \mathbb{R}^{n\times Nn} . \notag
\end{align}
By using the closed-loop dynamics under $(\hat u_1, \cdots, \hat u_N)$, we consider the SDE of $X^{(N)}$ and let $N\to \infty$.
This gives
the mean field limit state $\overline{X}$ as follows:
\begin{align}
\label{dbarX}
 & d \overline{X} =   ( A + G - B( \Theta + \Theta_1 ) ) \overline{X}   dt - B_0 (\Theta + \Theta_1 ) \overline{X}   d W_0 , \quad
t \geq 0 ,
\end{align}
where $\overline{X}(0) = \mu_0 $.
We denote the set of decentralized feedback strategies
\begin{align}
 \check{u}_i (t)
 = - \Theta(t) \mathbf{e}_i X(t) - \Theta_1(t) \overline{X}(t) , \quad
1 \leq i \leq N .
\label{checkui}
\end{align}

The state dynamics under the decentralized strategies \eqref{checkui} follows
\begin{align}
d X(t)  = &   ( \mathbf{A} X
 - \widehat{\mathbf{B}}( \widehat\Theta X + \widehat\Theta_1 \overline{X} ) )  dt
 - \sum_{i=1}^N   \mathbf{B}_i (\Theta X_i + \Theta_1 \overline{X})   d W_i  \notag \\
 &   - \mathbf{B}_0 \sum_{i=1}^N (\Theta X_i + \Theta_1 \overline{X}  )   d W_0 , \quad t\geq 0 , \notag
\end{align}
where the initial state $X(0)=(X_1^T(0), \cdots, X_N^T(0))^T$ is the same as in  \eqref{X}.

Below we evaluate the cost with more general initial conditions.
When all the $N$ players take the decentralized strategies \eqref{checkui}, the cost of  player $\mathcal{A}_i$ with initial condition $(X(t), \overline{X}(t))=(\mathbf{x}, \bar{x})$ is denoted by $\check{V}_i(t, \mathbf{x}, \bar{x})$, $t\in[0, T]$, $\mathbf{x} \in \mathbb{R}^{Nn}$, $\bar{x} \in \mathbb{R}^n$.
The Feynman--Kac formula~\cite[Sec. 1.3, 3.5]{P2009} gives the following  equation that $\check{V}_i$ satisfies:
\begin{align}
\label{HJBcheckV}
\begin{cases}
- \frac{\partial \check{V}_i}{\partial t}
 =  \frac{\partial^T \check{V}_i}{\partial \mathbf{x}}  ( \mathbf{A} \mathbf{x}
 - \widehat{\mathbf{B}}( \widehat\Theta \mathbf{x} + \widehat\Theta_1 \bar{x}  ) )
 + \frac{\partial^T \check{V}_i}{\partial \bar{x}}  (A+G - B( \Theta + \Theta_1 ) ) \bar{x}    \\
%%%
\hspace{1.3cm}
+ (1/2) ( \Theta \mathbf{I} \mathbf{x} + \mathbf{I} \widehat\Theta_1 \bar{x}   )^T
 \mathbf{B}_0^T \frac{\partial^2 \check{V}_i}{\partial \mathbf{x}^2 }
 \mathbf{B}_0 ( \Theta \mathbf{I} \mathbf{x} + \mathbf{I} \widehat\Theta_1 \bar{x}  )     \\
%%%
\hspace{1.3cm}   + (1/2) (  ( \Theta + \Theta_1 ) \bar{x} )^T  B_0^T\frac{\partial^2 \check{V}_i}{\partial \bar{x}^2}   B_0
( \Theta + \Theta_1 ) \bar{x}   \\
%%%
 \hspace{1.3cm}    + (1/2) \sum_{k=1}^N (  \Theta \mathbf{e}_k \mathbf{x} + \Theta_1 \bar{x}    )^T \mathbf{B}_k^T \frac{\partial^2 \check{V}_i}{\partial \mathbf{x}^2}
  \mathbf{B}_k (\Theta \mathbf{e}_k \mathbf{x} + \Theta_1 \bar{x} )    \\
%%%
\hspace{1.3cm}
+   (   \Theta \mathbf{I} \mathbf{x} + \mathbf{I} \widehat\Theta_1 \bar{x}  )^T
 \mathbf{B}_0^T \frac{\partial^2 \check{V}_i }{ \partial \mathbf{x} \partial \bar{x}}   B_0 ( \Theta + \Theta_1 ) \bar{x} \\
%%%
\hspace{1.3cm}  + ( \Theta \mathbf{e}_i \mathbf{x} + \Theta_1 \bar{x}  )^T R
 ( \Theta \mathbf{e}_i \mathbf{x} + \Theta_1 \bar{x}  )
+ \mathbf{x}^T \mathbf{Q}_i \mathbf{x} ,  \\
%%%
 \check{V}_i ( T, \mathbf{x} ) =  \mathbf{x}^T \mathbf{Q}_{if} \mathbf{x} .
\end{cases}
\end{align}
Assume $\check{V}_i(t, \mathbf{x}, \overline{x})$ takes the following form
\begin{align}
\label{checkVansatz}
 \check{V}_i(t, \mathbf{x} , \overline{x}) = \mathbf{x}^T
\check{\mathbf{P}}^i_1(t) \mathbf{x} + 2 \mathbf{x}^T
\check{\mathbf{P}}^i_{12} (t) \overline{x}  + \overline{x}^T
\check{\mathbf{P}}_2^i (t) \overline{x}  , \quad 1\leq i \leq N .
\end{align}
Substituting \eqref{checkVansatz} into \eqref{HJBcheckV}
gives the following system of ODEs for
$\check{\mathbf{P}}_1^i$, $\check{\mathbf{P}}_{12}^i$
and $\check{\mathbf{P}}_2^i$:
\begin{align}
& \begin{cases}
- \frac{d}{dt}\check{\mathbf{P}}_1^i =  \check{\mathbf{P}}_1^i
(\mathbf{A} -  \widehat{\mathbf{B}} \widehat\Theta )
+ ( \mathbf{A}   - \widehat{\mathbf{B}}  \widehat\Theta  )^T  \check{\mathbf{P}}_1^i
 + \mathbf{I}^T \Theta^T \mathbf{B}_0^T \check{\mathbf{P}}^i_1 \mathbf{B}_0 \Theta \mathbf{I}        \\
\hspace{1.55cm}
+ \sum_{k=1}^N \mathbf{e}_k^T \Theta^T \mathbf{B}_k^T
 \check{\mathbf{P}}_1^i \mathbf{B}_k \Theta \mathbf{e}_k
 + \mathbf{e}_i^T \Theta^T R \Theta \mathbf{e}_i + \mathbf{Q}_i ,   \\
\check{\mathbf{P}}_1^i (T) =  \mathbf{Q}_{if}  ,
\end{cases} \label{ODEcheckP1} \\
%%%%%%
& \begin{cases}
 -  \frac{d}{dt}\check{\mathbf{P}}_{12}^i =  -  \check{\mathbf{P}}_1^i \widehat{\mathbf{B}} \widehat\Theta_1 + (\mathbf{A} - \widehat{\mathbf{B}} \widehat\Theta )^T \check{\mathbf{P}}_{12}^i
 + \check{\mathbf{P}}_{12}^i (A+G - B(\Theta + \Theta_1 ) )   \\
\hspace{1.7cm}
 + \mathbf{I}^T \Theta^T \mathbf{B}_0^T \check{\mathbf{P}}_1^i
 \mathbf{B}_0 \mathbf{I}\widehat\Theta_1
 + \sum_{k=1}^N \mathbf{e}_k^T \Theta^T \mathbf{B}_k^T
 \check{\mathbf{P}}_1^i \mathbf{B}_k \Theta_1  \\
\hspace{1.7cm}
 + \mathbf{I}^T \Theta^T \mathbf{B}_0^T \check{\mathbf{P}}_{12}^i B_0 ( \Theta + \Theta_1 )
 + \mathbf{e}_i^T \Theta^T R \Theta_1 ,  \\
%%%%%
\check{\mathbf{P}}_{12}^i(T) = 0 ,
\end{cases} \label{ODEcheckP12} \\
%%%%%%%%%%%%%%
&\begin{cases}
 - \frac{d}{dt} \check{\mathbf{P}}_2^i =  - \check{\mathbf{P}}_{12}^{iT} \widehat{\mathbf{B}} \widehat\Theta_1 - \widehat\Theta_1^T \widehat{\mathbf{B}}^T \check{\mathbf{P}}_{12}^i
 + \check{\mathbf{P}}_2^i (A+G - B(\Theta + \Theta_1)) \\
\hspace{1.55cm}
+ (A+G - B(\Theta + \Theta_1))^T \check{\mathbf{P}}_2^i
 + ( \mathbf{B}_0 \mathbf{I} \widehat\Theta_1 )^T \check{\mathbf{P}}_1^i
 \mathbf{B}_0 \mathbf{I}\widehat\Theta_1    \\
\hspace{1.55cm}
+ \sum_{k=1}^N (\mathbf{B}_k \Theta_1 )^T \check{\mathbf{P}}_1^i \mathbf{B}_k \Theta_1
 + (\Theta + \Theta_1 )^T B_0^T \check{\mathbf{P}}_2^i B_0 (\Theta + \Theta_1 )  \\
\hspace{1.55cm}
  + (\mathbf{B}_0  \mathbf{I} \widehat{\Theta}_1 )^T \check{\mathbf{P}}_{12}^i
B_0 (\Theta + \Theta_1 )
 + ( B_0(\Theta + \Theta_1 ) )^T  \check{\mathbf{P}}_{12}^{iT}
 \mathbf{B}_0 \mathbf{I} \widehat{\Theta}_1   \\
\hspace{1.55cm}
+ \Theta_1^T R \Theta_1 , \\
%%%
\check{\mathbf{P}}^i_2(T) = 0 .
\end{cases} \label{ODEcheckP2}
\end{align}

\begin{remark}
\label{rmk:checkPsol}
 \eqref{ODEcheckP1}--\eqref{ODEcheckP2} is a first order linear ODE system and  admits a unique solution.
\end{remark}

We have the following submatrix partition of the matrices
$\check{\mathbf{P}}_1^i$, $\check{\mathbf{P}}_{12}^i$, $1\leq i\leq N$.
\begin{lemma}
\label{lm:checkPsubmat}
For \eqref{ODEcheckP1} and \eqref{ODEcheckP12}, the solution
$(\check{\mathbf{P}}^i_{1}, \check{\mathbf{P}}^i_{12})$,
$1\leq i \leq N$, has the representation
\begin{align}
\label{checkP1submat}
& \check{\mathbf{P}}^1_1
 = \begin{bmatrix}
\check\Pi_1^N & \check\Pi_2^N & \check\Pi_2^N & \cdots & \check\Pi_2^N \\
 \check\Pi_2^{NT} & \check\Pi_3^N & \check\Pi_4^N & \cdots & \check\Pi_4^N \\
\check\Pi_2^{NT} & \check\Pi_4^N & \check\Pi_3^N & \cdots & \check\Pi_4^N \\
\vdots & \vdots & \vdots & \ddots & \vdots \\
 \check\Pi_2^{NT} & \check\Pi_4^N & \check\Pi_4^N & \cdots & \check\Pi_3^N \\
\end{bmatrix} ,
\qquad
 \check{\mathbf{P}}_1^i = J_{1i}^T \check{\mathbf{P}}_1^1 J_{1i},
\quad \forall  2\leq i \leq N ,  \\
%%%
& \check{\mathbf{P}}_{12}^1 = \begin{bmatrix} \check{\Pi}_{11}^{NT} , \ \check{\Pi}_{12}^{NT} ,\ \cdots , \ \check{\Pi}_{12}^{NT}  \end{bmatrix}^T , \quad
\check{\mathbf{P}}_{12}^i = J_{1i}^T \check{\mathbf{P}}_{12}^1 ,
\quad \forall  2\leq i \leq N ,
\label{checkP12submat}
\end{align}
where $\check{\Pi}_1^N(t)$, $\check{\Pi}_3^N(t)$, $\check{\Pi}_4^N(t)\in \mathcal{S}^n$, and
$\check{\Pi}_2^N(t)$, $\check{\Pi}_{11}^N(t)$, $\check{\Pi}_{12}^N(t)\in \mathbb{R}^{n\times n}$.
\end{lemma}
\begin{proof}
The proof is similar to that of Lemma~\ref{lm:Psubmat} or \cite[Theorem 3]{HZ2020} and is omitted.
\end{proof}

We define the new variables:
\begin{align}
\begin{cases}
   \check\Lambda_1^N = \check\Pi_1^N , \hspace{0.4cm}
    \check\Lambda_2^N =N\check\Pi_2^N   , \hspace{0.42cm}
    \check\Lambda_3^N =  N^2\check\Pi_3^N  , \hspace{0.4cm}
   \check\Lambda_4^N =N^2\check\Pi_4^N   ,  \\
%%%%%
 \check\Lambda_{11}^N =\check\Pi_{11}^N  ,
\quad
   \check\Lambda_{12}^N  =  N\check\Pi_{12}^N  ,
\quad
   \check\Lambda_{22}^N = \check{\mathbf{P}}_2^i.
\end{cases} \label{checkPLam}
\end{align}

Substituting~\eqref{checkP1submat}--\eqref{checkP12submat} into \eqref{HJBcheckV}
and next converting into the new variables in  \eqref{checkPLam}, we derive
\begin{align}
& \begin{cases}
 - \frac{d}{dt}\check{\Lambda}_1^N = \check\Lambda_1^N ( A - B \Theta )
+ ( A  - B \Theta )^T \check\Lambda_1^N
 + \Theta^T ( R + B_1^T \check\Lambda_1^N B_1 ) \Theta  \\
\hspace{1.5cm}  + Q + \check{g}_1^N ,  \\
%%%
\check\Lambda_1^N(T) = (I - \Gamma_f^T/N) Q_f (I - \Gamma_f/N) ,
\end{cases}  \label{ODEcheckLam1} \\
%%%%%%%%
& \begin{cases}
  - \frac{d}{dt}\check\Lambda_2^N  =  \check\Lambda_1^N G
 + \check\Lambda_2^N ( A + G - B \Theta )
+ ( A  - B \Theta )^T \check\Lambda_2^N  - Q \Gamma
 + \check{g}_2^N , \\
%%%
\check\Lambda_2^N (T) =   - (I - \Gamma_f^T/N) Q_f \Gamma_f  ,
\end{cases} \label{ODEcheckLam2}
\end{align}
%%%%%%%%%%%
\begin{align}
& \begin{cases}
 - \frac{d}{dt}\check\Lambda_3^N =
 \Theta^T [ B_1^T \check\Lambda_3^N B_1 + B_0^T (\check\Lambda_1^N + \check\Lambda_2^N
+ \check\Lambda_2^{NT} + \check\Lambda_4^N ) B_0 ] \Theta \\
\hspace{1.5cm}
 + ( \check\Lambda_2^{NT}
+ \check\Lambda_4^N ) G
+ G^T ( \check\Lambda_2^N +  \check\Lambda_4^N )
+ \check\Lambda_3^N (A  -  B \Theta )   \\
\hspace{1.5cm}
+ ( A  -  B \Theta )^T \check\Lambda_3^N
+ \Gamma^T Q \Gamma + \check{g}_3^N ,  \\
%%%
\check\Lambda_3^N (T) = \Gamma_f^T Q_f \Gamma_f  ,
\end{cases} \label{ODEcheckLam3}
\end{align}
\begin{align}
& \begin{cases}
 - \frac{d}{dt}\check\Lambda_4^N =   \check\Lambda_2^{NT} G
+ G^T \check\Lambda_2^N + \check\Lambda_4^N (A + G - B \Theta ) + (A + G   - B \Theta  )^T \check\Lambda_4^N   \\
\hspace{1.65cm}
+ \Theta^T  B_0^T (\check\Lambda_1^N + \check\Lambda_2^N
+ \check\Lambda_2^{NT} + \check\Lambda_4^N ) B_0  \Theta
 + \Gamma^T Q \Gamma  + \check{g}_4^N   ,   \\
%%%
\check\Lambda_4^N(T) = \Gamma_f^T Q_f \Gamma_f ,
\end{cases} \label{ODEcheckLam4}
\end{align}
%%%%%%%%%
\begin{align}
& \begin{cases}
 - \frac{d}{dt}\check\Lambda_{11}^N =
  \check{\Lambda}_{11}^N (A+G - B(\Theta + \Theta_1 ))
+ ( A  - B \Theta )^T \check{\Lambda}_{11}^N \\
\hspace{1.6cm}
- ( \check\Lambda_1^N
 + \check\Lambda_2^N ) B \Theta_1
+ \Theta^T ( R + B_1^T \check\Lambda_1^N B_1 ) \Theta_1
  + \check{g}_{11}^N     ,   \\
%%%
 \check\Lambda_{11}^N(T) = 0 ,
\end{cases} \label{ODEcheckLam11}
\end{align}
%%%%%%%%%%%%%%
\begin{align}
& \begin{cases}
 - \frac{d}{dt}\check\Lambda_{12}^N
=  - ( \check\Lambda_2^{NT}  +  \check\Lambda_4^N  ) B \Theta_1
 +  ( A + G  - B \Theta   )^T \check\Lambda_{12}^N
 + G^T \check\Lambda_{11}^N    \\
\hspace{1.65cm}
 + \check\Lambda_{12}^N  (A+G - B (\Theta + \Theta_1 ) ) \\
\hspace{1.65cm}
 + \Theta^T B_0^T ( \check\Lambda_1^N
 + \check\Lambda_2^N + \check\Lambda_2^{ NT}
 + \check\Lambda_4^N ) B_0 \Theta_1     \\
\hspace{1.65cm}
 + \Theta^T B_0^T (\check\Lambda_{11}^N
 +  \check\Lambda_{12}^N  ) B_0 (\Theta + \Theta_1 )
 + \check{g}_{12}^N , \\
%%%
\check\Lambda_{12}^N(T) = 0 ,
\end{cases}  \label{ODEcheckLam12}
\end{align}
%%%%%%%%
\begin{align}
& \begin{cases}
 - \frac{d}{dt}\check\Lambda_{22}^N =
- ( \check{\Lambda}_{11}^N + \check{\Lambda}_{12}^N
 + \check\Lambda_{22}^N )^T B \Theta_1
 - \Theta_1^T B^T  ( \check\Lambda_{11}^N
 +  \check\Lambda_{12}^N + \check\Lambda_{22}^N )  \\
\hspace{1.65cm}
 + \check\Lambda_{22}^N (A+G - B \Theta )
 + (A+G - B \Theta )^T \check\Lambda_{22}^N     \\
 \hspace{1.65cm}
 + \Theta^T B_0^T \check\Lambda_{22}^N B_0 \Theta
+ \Theta^T B_0^T ( \check\Lambda_{11}^N
 + \check\Lambda_{12}^N
+ \check\Lambda_{22}^N )^T B_0 \Theta_1 \\
\hspace{1.65cm}
 + \Theta_1^T B_0^T ( \check\Lambda_{11}^N
 + \check\Lambda_{12}^N + \check\Lambda_{22}^N ) B_0 \Theta
+  \Theta_1^T ( R + B_1^T \check\Lambda_1^N  B_1 ) \Theta_1 \\
\hspace{1.65cm}
  + \Theta_1^T B_0^T (\check\Lambda_1^N + \check\Lambda_2^N
+ \check\Lambda_2^{NT} + \check\Lambda_4^N \\
\hspace{1.65cm}
+ \check\Lambda_{11}^N + \check\Lambda_{12}^N
+ \check\Lambda_{11}^{NT} +  \check\Lambda_{12}^{NT}
+ \check\Lambda_{22}^N ) B_0 \Theta_1
 + \check{g}_{22}^N  ,\\
%%%
\check\Lambda_{22}^N(T) = 0 .
\end{cases} \label{ODEcheckLam22}
\end{align}
The perturbation terms $\check{g}^N_1,\cdots, \check{g}^N_{22}$ are
listed in Appendix \ref{appendix:gdN}.

\begin{remark}
\label{rmk:solcheckLam}
Under Assumption~\ref{assm:solLam}, the system~\eqref{ODEcheckLam1}--\eqref{ODEcheckLam22} is a first order linear ODE system and admits a unique solution $(\check\Lambda_1^N, \cdots, \check\Lambda_{22}^N)$ on $[0, T]$.
\end{remark}

\begin{remark}
\label{rmk:bdcheckLam}
Let $\psi^N$ stand for any of the functions $\check\Lambda_1^N$, $\check\Lambda_2^N$, $\check\Lambda_3^N$, $\check\Lambda_4^N$,
$\check\Lambda_{11}^N$, $\check\Lambda_{12}^N$ and $\check\Lambda_{22}^N$. Due to the bounded coefficients in the ODE system
\eqref{ODEcheckLam1}--\eqref{ODEcheckLam22}, $\sup_{N\geq \hat{N}_1, 0 \leq t \leq T} |\psi^N| \leq C$ for some fixed constant $C$.
\end{remark}

\begin{remark}
\label{rmk:checkgO(1/N)}
Let $h^N$ stand for any of the functions $\check{g}_1^N$, $\check{g}_2^N$, $\check{g}_3^N$, $\check{g}_4^N$, $\check{g}_{11}^N$, $\check{g}_{12}^N$ and $\check{g}_{22}^N$. Then $\sup_{t\in [0, T]} |h^N(t)|=O(1/N)$.
\end{remark}

Let $(\check\Lambda_1^N, \cdots, \check\Lambda_{22}^N)$ be obtained from \eqref{ODEcheckLam1}--\eqref{ODEcheckLam22}.
 By substituting \eqref{checkP1submat} into \eqref{checkVansatz}, which is further expressed in terms of $(\check\Lambda_1^N, \cdots, \check\Lambda_{22}^N)$ via \eqref{checkPLam},    we obtain an explicit representation of a player's cost when all the players take the set of decentralized strategies
$(\check u_1, \cdots, \check u_N)$  in \eqref{checkui}.
The cost of player $\mathcal{A}_i$ is
\begin{align}
   J_i(\check{u}_i, \check{u}_{-i})
=& \mathbb{E}\Big[  \check{V}_i(0, X(0), \overline{X}(0))
  \Big]   \notag \\
= & \mathbb{E} \Big[ X^T(0)   \check{\mathbf{P}}_1^i(0)  X(0)
 + 2 X^T(0)  \check{\mathbf{P}}_{12}^i(0)  \overline{X}(0)
 + \overline{X}^T(0)  \check{\mathbf{P}}_2^i(0)  \overline{X}(0) \Big] .  \label{checkJiP}
\end{align}
Denote $\mathcal{N}_{-i}=\{1, \cdots, N\}\setminus\{i\}$.
Under Assumption \ref{assm:initialX}, the first term on the right hand side of \eqref{checkJiP} is
\begin{align}
& \mathbb{E} \Big[ X^T(0)  \check{\mathbf{P}}_1^i(0)  X(0) \Big]  \notag \\
= &  \mathbb{E}\Big[ X_i^T(0)  \check\Lambda_1^N (0) X_i(0)
+ (2/N) \sum_{j\in \mathcal{N}_{-i} } X_i^T(0) \check\Lambda_2^N(0) X_j(0)
 \notag \\
& + \frac{1}{N^2}\sum_{j\in \mathcal{N}_{-i}  } X_j^T(0) \check\Lambda_3^N(0) X_j(0)
 + \frac{1}{N^2}\sum_{{ j, k \in \mathcal{N}_{-i},  j \neq k  } } X_j^T(0) \check\Lambda_4^N(0)  X_k(0) \Big] \notag \\
%%%%%
= &  \Tr[  \check\Lambda_1^N(0) \Sigma_0^i ]
 + (1/N^2) \sum_{j\in \mathcal{N}_{-i}  } \Tr [ \check\Lambda_3^N(0) \Sigma_0^j]   \notag \\
& + \mu_0^T [   \check\Lambda_1^N(0) + \check\Lambda_2^N(0) + \check\Lambda_2^{NT}(0) + \check\Lambda_4^N(0) ] \mu_0 \notag  \\
& + \mu_0^T [  - (\check\Lambda_2^N(0) + \check\Lambda_2^{NT}(0)) /N + \check\Lambda_3^N(0) (N-1)/N^2   \notag \\ &\qquad
+ \check\Lambda_4^N(0)  (2-3N)/N  ]  \mu_0.
\label{XcheckP1}
\end{align}
The second term on the right hand side of \eqref{checkJiP} can be written as
\begin{align}
 & 2 \mathbb{E} \Big[  X^T(0)  \check{\mathbf{P}}_{12}^i(0)  \overline{X}(0) \Big] \label{XcheckP12}  \\
=&  \mu_0^T [  \check\Lambda_{11}^{N}(0)
 +   \check\Lambda_{11}^{NT}(0)
 + \check\Lambda_{12}^{N}(0) + \check\Lambda_{12}^{NT}(0)  ] \mu_0
 %\notag \\ &
  - \mu_0^T [ \check\Lambda_{12}^N(0)
+ \check\Lambda_{12}^{NT}(0)   ] \mu_0 /N .  \notag
\end{align}
The third term on the right hand side of \eqref{checkJiP} can be written as
\begin{align}
 \mathbb{E} [ \overline{X}^T(0)  \check{\mathbf{P}}_2^i(0)  \overline{X}(0) ]
 = \mu_0^T  \check\Lambda_{22}^N(0)  \mu_0 .  \label{XcheckP2}
\end{align}
Denote
\begin{align}
& \check{Y}^N :=   \check\Lambda_1^N + \check\Lambda_2^N
 + \check\Lambda_2^{NT}  + \check\Lambda_4^N
+ \check\Lambda_{11}^N
+ \check\Lambda_{11}^{NT} + \check\Lambda_{12}^N
+ \check\Lambda_{12}^{NT}
+ \check\Lambda_{22}^N .
\label{checkYN}
\end{align}
Substituting \eqref{XcheckP1}, \eqref{XcheckP12} and \eqref{XcheckP2} into \eqref{checkJiP} gives
\begin{align}
   J_i(\check{u}_i, \check{u}_{-i})
=& \mu_0^T   \check{Y}^N(0)  \mu_0
+  \Tr[  \check\Lambda_1^N(0) \Sigma_0^i ]
 + (1/N^2) \sum_{j\in \mathcal{N}_{-i} } \Tr [ \check\Lambda_3^N(0) \Sigma_0^j]   \notag \\
& + \mu_0^T \{  - (\check\Lambda_2^N(0) + \check\Lambda_2^{NT}(0))/N  +  \check\Lambda_3^N(0) (N-1)/N^2 \notag \\
 & +  \check\Lambda_4^N(0)  (2-3N)/N^2
  - (   \check\Lambda_{12}^N(0)  + \check\Lambda_{12}^{NT}(0)  ) /N
\}  \mu_0 .   \label{checkJiLam}
\end{align}

%%%%%%%%%%%%%%%%%%%%%%%%%%%%%
%%%%%%%%%%%%%%%%%%%%%%%%%%%%%
\section{Decentralized  $O(1/N)$-Nash equilibrium strategies}
\label{sec:O(1/N)Nash}

In this section we show that the set of decentralized strategies in
\eqref{checkui} has an $O(1/N)$-Nash equilibrium property. More precisely, when the game \eqref{Xi}--\eqref{Ji} is asymptotically solvable and all  other players take the decentralized strategies \eqref{checkui}, the extra benefit that a player obtains by unilaterally deviating from the strategy
\eqref{checkui} is at most $O(1/N)$.
%%%%%%%%%%%%%%%%
\begin{theorem}
\label{thm:O(1/N)Nash}
Under Assumptions \ref{assm:solLam} and \ref{assm:initialX},
the set of decentralized strategies $(\check{u}_1, \cdots, \check{u}_N)$  given by \eqref{checkui} is an $O(1/N)$-Nash equilibrium  of the Nash game \eqref{Xi}--\eqref{Ji}, i.e.,
\begin{align}
 J_i(\check{u}_i, \check{u}_{-i})  \leq J_i(u_i, \check{u}_{-i}) + O(1/N), \quad \forall 1\le i\le N,\label{epsNash}
 \end{align}
  where $u_i$ is any admissible control under CLPS information such that the closed-loop system under $(u_i,\hat u_{-i})$ has  a well defined solution.  \end{theorem}

We will  prove Theorem~\ref{thm:O(1/N)Nash} after some technical preparations.
Without loss of generality, we prove \eqref{epsNash} for player $\mathcal{A}_1$.
Suppose that players $\mathcal{A}_i$, $2\leq i \leq N$, use decentralized strategies  given by \eqref{checkui}.
Player $\mathcal{A}_1$ seeks its best response strategy $u_1^b$
with respect to $\check{u}_{-1}$ so that  $J_1(u_1^b, \check{u}_{-1})= \inf_{u_1} J_1(u_1,\hat u_{-1}) $, where $J_1$ is defined by \eqref{Ji}. This leads to the optimal control problem with dynamics
\begin{align}
 d X(t)  = &  \Big[ \mathbf{A} X  + \widehat{\mathbf{B}}_1 u_1
- \widehat{\mathbf{B}}_{-1} ( \widehat\Theta X
 + \widehat\Theta_1 \overline{X} ) \Big]  dt + \mathbf{B}_1 u_1   d W_1 \notag \\
%%%
& - \sum_{i=2}^N  \mathbf{B}_i ( \Theta \mathbf{e}_i X + \Theta_1 \overline{X} ) d W_i
  + \mathbf{B}_0 \Big( u_1 -  \sum_{i=2}^N (\Theta \mathbf{e}_i X + \Theta_1 \overline{X}  )   \Big) d W_0 , \notag
\end{align}
where we denote $\widehat{\mathbf{B}}_{-1} = (0, \widehat{\mathbf{B}}_2, \cdots , \widehat{\mathbf{B}}_N)$ and
the mean field limit state $\overline{X}$ follows the dynamics
\eqref{dbarX}. The best response $u_1^b$ is to be determined.

We employ a dynamic programming approach to solve player $\mathcal{A}_1$'s optimal control problem.
Let $V_1^b(t, \mathbf{x}, \bar{x})$ be the value function of $\mathcal{A}_1$  with initial state $(X(t), \overline{X}(t))=(\mathbf{x}, \bar{x})$, associated with the cost  $J_1(u_1, \check u_{-1})$.
Now $V_1^b(t, \mathbf{x}, \bar{x})$ is formally solved from  the following dynamic programming equation:
\begin{align}
\label{HJBV1d}
& - \frac{\partial V_1^b}{\partial t}
 =  \min_{u_1 \in \mathbb{R}^{n_1}} \left[ \frac{\partial^T V_1^b}{\partial \mathbf{x}}
 ( \mathbf{A} \mathbf{x}  + \widehat{\mathbf{B}}_1 u_1
 - \widehat{\mathbf{B}}_{-1} ( \widehat\Theta \mathbf{x}
 + \widehat\Theta_1 \bar{x} ) ) \right. \\
&\qquad\qquad + \frac{\partial^T V_1^b}{\partial \bar{x}}  (A+G - B(\Theta + \Theta_1 ) ) \bar{x}  \notag  \\
%%%
&\qquad\qquad  \left. + \frac{1}{2} \Big(  u_1 - \sum_{i=2}^N (\Theta \mathbf{e}_i \mathbf{x} + \Theta_1 \bar{x}  )    \Big) ^T \mathbf{B}_0^T \frac{\partial^2 V_1^b}{\partial \mathbf{x}^2 } \mathbf{B}_0 (  u_1 -  \sum_{i=2}^N (\Theta \mathbf{e}_i \mathbf{x} + \Theta_1 \bar{x}  )     )   \right.  \notag  \\
%%%
 & \qquad\qquad \left.  + \frac{1}{2}  ( \mathbf{B}_1 u_1  )^T \frac{\partial^2 V_1^b}{\partial \mathbf{x}^2}  \mathbf{B}_1 u_1
+ \frac{1}{2} \sum_{i=2}^N ( \mathbf{B}_i (\Theta \mathbf{e}_i \mathbf{x} + \Theta_1 \bar{x} )  )^T \frac{\partial^2 V_1^b}{\partial \mathbf{x}^2} \mathbf{B}_i (\Theta \mathbf{e}_i \mathbf{x} + \Theta_1 \bar{x} )   \right. \notag \\
%%%
 &\qquad\qquad  \left.
 + \frac{1}{2} (   B_0 ( \Theta + \Theta_1 ) \bar{x} )^T \frac{\partial^2 V_1^b}{\partial \bar{x}^2}   B_0 ( \Theta + \Theta_1 ) \bar{x}
 + \mathbf{x}^T \mathbf{Q}_1 \mathbf{x} + u_1^{T} R u_1 \right.  \notag \\
& \qquad\qquad \left.
 -   \Big( u_1  -  \sum_{i=2}^N (\Theta \mathbf{e}_i \mathbf{x} + \Theta_1 \bar{x} )  \Big)^T \mathbf{B}_0^T   \frac{\partial^2 V_1^b}{ \partial \mathbf{x} \partial \bar{x}}   B_0 (\Theta + \Theta_1 ) \bar{x}  \right]  ,
\notag \\
%%%
 & V_1^b ( T, \mathbf{x} , \bar{x} ) =  \mathbf{x}^T \mathbf{Q}_{1f} \mathbf{x} , \quad t \in [0, T], \ \mathbf{x} \in \mathbb{R}^{Nn}, \
 \bar{x} \in \mathbb{R}^n . \notag
\end{align}
The first order condition with respect to $u_1$ gives
\begin{align}
 u_1^b = &
 - \frac{1}{2}\Big( R +  \frac{1}{2} \mathbf{B}_1^T \frac{\partial^2 V_1^b}{\partial \mathbf{x}^2} \mathbf{B}_1 + \frac{1}{2}\mathbf{B}_0^T \frac{\partial^2 V_1^b}{\partial \mathbf{x}^2} \mathbf{B}_0 \Big)^{-1} \cdot \label{u1dV} \\
&  \Big[ \widehat{\mathbf{B}}_1^T \frac{\partial V_1^b}{\partial \mathbf{x}}
 - \mathbf{B}_0^T \frac{\partial^2 V_1^b}{\partial \mathbf{x}^2} \mathbf{B}_0
 \sum_{i=2}^N ( \Theta \mathbf{e}_i \mathbf{x} + \Theta_1 \bar{x} )
 - \mathbf{B}_0^T \frac{\partial^2 V_1^b}{\partial \mathbf{x} \partial \bar{x}}  B_0 ( \Theta + \Theta_1 ) \bar{x}  \Big] . \notag
\end{align}
We substitute \eqref{u1dV} into \eqref{HJBV1d} to obtain
\begin{align}
\label{HJBV1d2}
 -\frac{\partial V_1^b}{\partial t} =&
 -  \Big[ \widehat{\mathbf{B}}_1^T \frac{\partial V_1^b}{\partial \mathbf{x}} - \mathbf{B}_0^T \frac{\partial^2 V_1^b}{\partial \mathbf{x}^2} \mathbf{B}_0
\sum_{i=2}^N (\Theta \mathbf{e}_i \mathbf{x} + \Theta_1 \bar{x} )   - \mathbf{B}_0^T \frac{\partial^2 V_1^b}{\partial \mathbf{x} \partial \bar{x}}  B_0(\Theta + \Theta_1) \bar{x}  \Big]^T \cdot   \\
 &\quad \frac{1}{4}\Big(  R +  \frac{1}{2} \mathbf{B}_1^T \frac{\partial^2 V_1^b}{\partial \mathbf{x}^2} \mathbf{B}_1 + \frac{1}{2} \mathbf{B}_0^T \frac{\partial^2 V_1^b}{\partial \mathbf{x}^2} \mathbf{B}_0 \Big)^{-1} \cdot \notag \\
 &\quad \Big[ \widehat{\mathbf{B}}_1^T \frac{\partial V_1^b}{\partial \mathbf{x}}
 - \mathbf{B}_0^T \frac{\partial^2 V_1^b}{\partial \mathbf{x}^2} \mathbf{B}_0 \sum_{i=2}^N
 (\Theta \mathbf{e}_i \mathbf{x} + \Theta_1 \bar{x} )
 - \mathbf{B}_0^T \frac{\partial^2 V_1^b}{\partial \mathbf{x} \partial \bar{x}}
 B_0(\Theta + \Theta_1 ) \bar{x}    \Big] \notag \\
%%%
 & + \frac{\partial^T V_1^b }{\partial \mathbf{x}} (\mathbf{A} \mathbf{x} - \sum_{i=2}^N \widehat{\mathbf{B}}_i (\Theta \mathbf{e}_i \mathbf{x} + \Theta_1 \bar{x} ) )
 + \frac{\partial^T V_1^b}{\partial \bar{x}}  (A+G - B ( \Theta + \Theta_1 ) ) \bar{x}   \notag \\
&  + \frac{1}{2} \Big( \sum_{i=2}^N (\Theta \mathbf{e}_i \mathbf{x} + \Theta_1 \bar{x} )    \Big)^T \mathbf{B}_0^T \frac{\partial^2 V_1^b}{\partial \mathbf{x}^2} \mathbf{B}_0 \sum_{i=2}^N
 (\Theta \mathbf{e}_i \mathbf{x} + \Theta_1 \bar{x} )    \notag \\
& + \frac{1}{2} \sum_{i=2}^N ( \Theta \mathbf{e}_i \mathbf{x} + \Theta_1 \bar{x}  )^T \mathbf{B}_i^T \frac{\partial^2 V_1^b }{\partial \mathbf{x}^2} \mathbf{B}_i
( \Theta \mathbf{e}_i \mathbf{x} + \Theta_1 \bar{x}  )  + \mathbf{x}^T \mathbf{Q}_1 \mathbf{x}   \notag \\
%%%
& + \frac{1}{2} ( B_0 ( \Theta + \Theta_1 ) \bar{x} )^T
 \frac{\partial^2 V_1^b}{\partial \bar{x}^2}  B_0 (\Theta + \Theta_1 ) \bar{x}
\notag \\
&  +  \Big(  \mathbf{B}_0 \sum_{i=2}^N ( \Theta \mathbf{e}_i \mathbf{x} + \Theta_1 \bar{x} )  \Big)^T
 \frac{\partial^2 V_1^b}{ \partial \mathbf{x} \partial \bar{x}}   B_0 ( \Theta + \Theta_1 ) \bar{x}  ,  \notag \\
%%%
 V_1^b ( T, \mathbf{x} , \bar{x} ) = & \mathbf{x}^T \mathbf{Q}_{1 f} \mathbf{x} . \notag
\end{align}
Assume $V_1^b$ takes the form
\begin{align}
V_1^b(t, \mathbf{x} , \bar{x}) = \mathbf{x}^T \mathbf{P}_1^b(t) \mathbf{x} + 2 \mathbf{x}^T
\mathbf{P}_{12}^b(t) \bar{x}  + \bar{x}^T \mathbf{P}_2^b(t) \bar{x}    .
\label{V1dansatz}
\end{align}
We denote $\mathbf{I}_{-1} = (0, I_n, \cdots, I_n) \in \mathbb{R}^{n \times N n}$, and substitute \eqref{V1dansatz} into \eqref{HJBV1d2} to obtain ODEs for $\mathbf{P}_1^b$, $\mathbf{P}_{12}^b$ and
$\mathbf{P}_2^b$:
\begin{align}
\label{ODEP1d}
\begin{cases}
 - \dot{\mathbf{P}}_1^b
=  - ( \widehat{\mathbf{B}}_1^T \mathbf{P}_1^b - \mathbf{B}_0^T \mathbf{P}_1^b \mathbf{B}_0 \Theta \mathbf{I}_{-1} )^T
 ( R + \mathbf{B}_1^T \mathbf{P}_1^b \mathbf{B}_1 + \mathbf{B}_0^T \mathbf{P}_1^b \mathbf{B}_0 )^{-1} \cdot \\
\qquad\qquad
( \widehat{\mathbf{B}}_1^T \mathbf{P}_1^b - \mathbf{B}_0^T \mathbf{P}_1^b \mathbf{B}_0 \Theta \mathbf{I}_{-1} )
 + \mathbf{P}_1^b ( \mathbf{A} - \widehat{\mathbf{B}}_{-1} \widehat\Theta )       \\
\hspace{1.2cm}
 + ( \mathbf{A} - \widehat{\mathbf{B}}_{-1} \widehat\Theta )^T
 \mathbf{P}_1^b
 + ( \mathbf{B}_0 \Theta \mathbf{I}_{-1} )^T \mathbf{P}_1^b   \mathbf{B}_0 \Theta \mathbf{I}_{-1}    \\
 \hspace{1.2cm}
 + \sum_{k=2}^N ( \mathbf{B}_k \Theta \mathbf{e}_k )^T
 \mathbf{P}_1^b ( \mathbf{B}_k \Theta \mathbf{e}_k )
 + \mathbf{Q}_1 , \\
%%%
\mathbf{P}_1^b(T) = \mathbf{Q}_1, \quad
R + \mathbf{B}_1^T \mathbf{P}_1^b(t) \mathbf{B}_1 + \mathbf{B}_0^T \mathbf{P}_1^b(t) \mathbf{B}_0 >0 , \ \forall  t \in [0, T]   ,
\end{cases}
\end{align}
\begin{align}
\label{ODEP12d}
\begin{cases}
  - \dot{\mathbf{P}}_{12}^b =
  (\mathbf{A} - \widehat{\mathbf{B}}_{-1} \widehat\Theta )^T
 \mathbf{P}_{12}^b  + \mathbf{P}_{12}^b (A+G - B(\Theta + \Theta_1 ) )  \\
\hspace{1.25cm}
 +  (\mathbf{B}_0 \Theta \mathbf{I}_{-1} )^T \mathbf{P}_1^b \mathbf{B}_0 \mathbf{I}_{-1}\widehat\Theta_1
- \mathbf{P}_1^b \widehat{\mathbf{B}}_{-1} \widehat\Theta_1
 \\
 \hspace{1.25cm}
  + \sum_{k=2}^N ( \mathbf{B}_k \Theta \mathbf{e}_k )^T \mathbf{P}_1^b \mathbf{B}_k \Theta_1
 + (\mathbf{B}_0 \Theta \mathbf{I}_{-1} )^T \mathbf{P}_{12}^b B_0 (\Theta + \Theta_1 )  \\
\hspace{1.25cm}
  - ( \widehat{\mathbf{B}}_1^T \mathbf{P}_1^b - \mathbf{B}_0^T \mathbf{P}_1^b \mathbf{B}_0 \Theta \mathbf{I}_{-1} )^T (R + \mathbf{B}_1^T \mathbf{P}_1^b \mathbf{B}_1 + \mathbf{B}_0^T \mathbf{P}_1^b \mathbf{B}_0 )^{-1} \cdot   \\
\qquad\qquad\quad
[ \widehat{\mathbf{B}}_1^T \mathbf{P}_{12}^b - \mathbf{B}_0^T \mathbf{P}_1^b \mathbf{B}_0 \mathbf{I}_{-1} \widehat\Theta_1
 - \mathbf{B}_0^T \mathbf{P}_{12}^b B_0( \Theta + \Theta_1 )  ]  ,           \\
%%%
\mathbf{P}_{12}^b(T) = 0 ,
\end{cases}
\end{align}
\begin{align}
\label{ODEP2d}
\begin{cases}
  - \dot{\mathbf{P}}_2^b =  - [ \widehat{\mathbf{B}}_1^T \mathbf{P}_{12}^b -  \mathbf{B}_0^T \mathbf{P}_1^b \mathbf{B}_0 \mathbf{I}_{-1}\widehat\Theta_1 - \mathbf{B}_0^T \mathbf{P}_{12}^b B_0 ( \Theta + \Theta_1 ) ]^T \cdot \\
\qquad\qquad
  (R + \mathbf{B}_1^T \mathbf{P}_1^b \mathbf{B}_1 + \mathbf{B}_0^T
 \mathbf{P}_1^b \mathbf{B}_0 )^{-1} \cdot  \\
\qquad\qquad
  [ \widehat{\mathbf{B}}_1^T \mathbf{P}_{12}^b -  \mathbf{B}_0^T \mathbf{P}_1^b \mathbf{B}_0 \mathbf{I}_{-1} \widehat\Theta_1 - \mathbf{B}_0^T
 \mathbf{P}_{12}^b B_0 ( \Theta + \Theta_1 ) ] \\
\hspace{1.2cm}
 - \mathbf{P}_{12}^{bT} \widehat{\mathbf{B}}_{-1} \widehat\Theta_1
 - \widehat\Theta_1^T \widehat{\mathbf{B}}_{-1}^T \mathbf{P}_{12}^b
 +  (\mathbf{B}_0 \mathbf{I}_{-1} \widehat\Theta_1 )^T \mathbf{P}_1^b \mathbf{B}_0 \mathbf{I}_{-1} \widehat\Theta_1         \\
 \hspace{1.2cm}
  + \mathbf{P}_2^b (A + G - B(\Theta + \Theta_1 ) ) + (A + G - B ( \Theta + \Theta_1 ) )^T \mathbf{P}_2^b    \\
\hspace{1.2cm}
 + \sum_{k=2}^N ( \mathbf{B}_k \Theta_1 )^T \mathbf{P}_1^b (\mathbf{B}_k \Theta_1 )
 + (\Theta + \Theta_1 )^T B_0^T \mathbf{P}_2^b B_0 (\Theta + \Theta_1 )   \\
 \hspace{1.2cm}
 +   ( \mathbf{B}_0 \mathbf{I}_{-1} \widehat\Theta_1  )^T \mathbf{P}_{12}^b B_0 ( \Theta + \Theta_1 )
 + (\Theta + \Theta_1 )^T B_0^T \mathbf{P}_{12}^{bT} \mathbf{B}_0 \mathbf{I}_{-1} \widehat\Theta_1   ,    \\
%%%
\mathbf{P}_2^b(T) = 0 .
\end{cases}
\end{align}

\begin{proposition}\label{prop:br}
Suppose that Assumption~\ref{assm:solLam} holds and that \eqref{ODEP1d} has a solution $\mathbf{P}_1^b$ on $[0, T]$.
Then we may uniquely solve \eqref{ODEP12d}--\eqref{ODEP2d}, and the best response strategy for $\mathcal{A}_1$ is
\begin{align}
\label{u1dP}
 u_1^b(t) = &
 -  \Big( R +  \mathbf{B}_1^T \mathbf{P}_1^b \mathbf{B}_1
 + \mathbf{B}_0^T \mathbf{P}_1^b \mathbf{B}_0 \Big)^{-1}
 \Big[ \widehat{\mathbf{B}}_1^T (  \mathbf{P}_1^b X(t) +  \mathbf{P}_{12}^b \overline{X}(t) ) \\
&\quad   - \mathbf{B}_0^T  \mathbf{P}_1^b \mathbf{B}_0
 \sum_{i=2}^N ( \Theta \mathbf{e}_i X(t) + \Theta_1 \overline{X}(t) )
 - \mathbf{B}_0^T  \mathbf{P}_{12}^b  B_0
 ( \Theta + \Theta_1 ) \overline{X}(t)  \Big] . \notag
\end{align}
\end{proposition}
\begin{proof}
If \eqref{ODEP1d} admits a (unique) solution $\mathbf{P}_1^b$ on $[0, T]$, then we can substitute $\mathbf{P}_1^b$ into
\eqref{ODEP12d} and solve a first order linear ODE for a unique
$\mathbf{P}_{12}^b$. Given $(\mathbf{P}_1^b, \mathbf{P}_{12}^b)$, $\mathbf{P}_2^b$ is again solved from a linear ODE.
Note that the LQ optimal control problem of player $\mathcal{A}_1$ has its Riccati equation given by \eqref{ODEP1d}--\eqref{ODEP2d}.
It then follows from \cite[Theorem 6.6.1]{YZ1999} that player $\mathcal{A}_1$'s optimal control problem is solvable with the optimal control  given by \eqref{u1dP}.
\end{proof}

 We will later show that for all sufficiently large $N$, \eqref{ODEP1d} indeed has a solution on $[0, T]$ (see Lemma \ref{lm:exissolP1P2d}). The next lemma is
 parallel to  Lemma~\ref{lm:Psubmat}.

\begin{lemma}
\label{lm:Pdsubmat}
Suppose \eqref{ODEP1d} has a solution $\mathbf{P}_1^b$ on $[0,T]$.
Then for \eqref{ODEP1d} and \eqref{ODEP12d}, $\mathbf{P}_1^b$ and $\mathbf{P}_{12}^b$ have the representations
\begin{align}
 & \mathbf{P}_1^b
 = \begin{bmatrix}
\Pi_1^{b N} & \Pi_2^{b N} & \Pi_2^{b N} & \cdots & \Pi_2^{b N} \\
( \Pi_2^{b N})^T &  \Pi_3^{b N}  &  \Pi_4^{b N}  & \cdots &  \Pi_4^{b N}  \\
( \Pi_2^{b N} )^T & \Pi_4^{b N} & \Pi_3^{b N} & \cdots & \Pi_4^{b N} \\
\vdots & \vdots & \vdots & \ddots & \vdots \\
( \Pi_2^{b N} )^T & \Pi_4^{b N} & \Pi_4^{b N} & \cdots & \Pi_3^{b N} \\
\end{bmatrix} , \label{P1dsubmat}\\
& \mathbf{P}_{12}^b = \begin{bmatrix} (\Pi_{11}^{b N})^T , \
 ( \Pi_{12}^{b N} )^T , \ \cdots , \ (\Pi_{12}^{b N})^T \end{bmatrix}^T ,
 \label{P12dsubmat}
\end{align}
where $\Pi^{bN}_1(t)$, $\Pi^{bN}_3(t)$, $\Pi^{bN}_4(t) \in \mathcal{S}^{ n}$, and $\Pi^{bN}_2(t)$, $\Pi^{bN}_{11}(t)$,  $\Pi^{bN}_{12}(t) \in \mathbb{R}^{n\times n}$.
\end{lemma}
\begin{proof}
The proof is similar to that of Lemma~\ref{lm:Psubmat}, and is thus omitted here.
\end{proof}

%%%%%%%%%%
We define new variables:
\begin{align}
\begin{cases}
     \Lambda_1^{b N} = \Pi_1^{bN}   , \quad
     \Lambda_2^{b N} = N \Pi_2^{bN}    , \quad
      \Lambda_3^{b N} = N^2 \Pi_3^{bN}  , \quad
      \Lambda_4^{b N} = N^2 \Pi_4^{bN} ,   \\
 %%%%%%%%%%%%
  \Lambda_{11}^{b N} =  \Pi_{11}^{bN}   , \quad
     \Lambda_{12}^{bN} = N \Pi_{12}^{bN}  , \quad
  \Lambda_{22}^{b N} =\mathbf{P}_2^b   ,
\end{cases} \label{PdLam}
\end{align}
and suppose \eqref{ODEP1d} has a solution $\mathbf{P}_1^b$ on $[0, T]$.
We substitute \eqref{P1dsubmat} and \eqref{P12dsubmat} into
\eqref{ODEP1d}--\eqref{ODEP2d} and take a change of variables by \eqref{PdLam} to obtain (under the additional condition that $R+B_1^T \Lambda_1^{b N}(t) B_1 > 0$)  the following ODEs:
\begin{align}
 & \begin{cases}
  \dot{\Lambda}_1^{b N} =   \Lambda_1^{b N} B
 ( \mathcal{R}_1 ( \Lambda_1^{b N} )  )^{-1} B^T \Lambda_1^{b N}
 - \Lambda_1^{b N } A  -  A^T \Lambda_1^{b N }    \\
  \hspace{1.1cm}
- Q  + g_1^{bN}  ,  \\
%%%
\Lambda_1^{b N}(T) =   (I - \Gamma_f^T/N) Q_f (I - \Gamma_f/N) , \\
 R+B_1^T \Lambda_1^{b N}(t) B_1 > 0 , \quad \forall  t\in[0, T] ,
\end{cases} \label{ODELam1dN} \\
%%%%%%%%%%%%%
 & \begin{cases}
  \dot\Lambda_2^{b N}  =  \Lambda_1^{b N} B
 ( \mathcal{R}_1 ( \Lambda_1^{b N} ) )^{-1} B^T\Lambda_2^{b N}
 - ( \Lambda_1^{b N} +  \Lambda_2^{b N} ) G  \\
 \hspace{1.1cm}
 - A^T  \Lambda_2^{b N} - \Lambda_2^{b N}  ( A -   B \Theta )
 + Q \Gamma  + g_2^{bN}  , \\
%%%
\Lambda_2^{bN}(T) =  - (I - \Gamma_f^T/N) Q_f \Gamma_f ,
\end{cases} \label{ODELam2dN}
\end{align}
%%%%%%%%%%%%%
\begin{align}
& \begin{cases}
\label{ODELam3dN}
  \dot\Lambda_3^{bN} =    (\Lambda_2^{bN})^T
B ( \mathcal{R}_1( \Lambda_1^{bN} ) )^{-1} B^T \Lambda_2^{b N}
 - (  \Lambda_2^{bN} + \Lambda_4^{bN} )^T G \\
\hspace{1.1cm}
 - G^T (\Lambda_2^{bN} + \Lambda_4^{bN} )
   - \Lambda_3^{b N} (A- B \Theta)
 - ( A - B \Theta )^T \Lambda_3^{b N} \\
\hspace{1.1cm}
 - \Theta^T B_0^T ( \Lambda_1^{b N} + \Lambda_2^{b N}
 + (\Lambda_2^{b N})^T  + \Lambda_4^{bN} ) B_0   \Theta     \\
\hspace{1.1cm}
 - \Theta^T  B_1^T \Lambda_3^{b N} B_1 \Theta
 - \Gamma^T Q \Gamma   + g_3^{bN}  , \\
%%%
\Lambda_3^{b N}(T)  =   \Gamma_f^T Q_f \Gamma_f   ,
\end{cases}
\end{align}
\begin{align}
& \begin{cases}
  \dot\Lambda_4^{b N} =   (\Lambda_2^{b N})^T
 B ( \mathcal{R}_1( \Lambda_1^{bN} ) )^{-1} B^T \Lambda_2^{b N}
 - (\Lambda_2^{b N})^T G - G^T \Lambda_2^{b N} \\
\hspace{1.1cm}
- \Lambda_4^{b N} (A+G -  B \Theta )
- (A+G -  B \Theta )^T  \Lambda_4^{b N} \\
\hspace{1.1cm}
 - \Theta^T B_0^T (\Lambda_1^{b N} + \Lambda_2^{b N}
 +  (\Lambda_2^{b N})^T + \Lambda_4^{b N} ) B_0 \Theta \\
\hspace{1.1cm}
 - \Gamma^T Q \Gamma + g_4^{bN}   ,  \\
%%%%
 \Lambda_4^{b N} (T) =  \Gamma_f^T Q_f \Gamma_f   ,
\end{cases} \label{ODELam4dN} \\
%%%%%%%%%%%%%
& \begin{cases}
  \dot\Lambda_{11}^{b N} =   \Lambda_1^{b N}  B
 ( \mathcal{R}_1 (\Lambda_1^{bN} ) )^{-1} B^T \Lambda_{11}^{b N}
 + \Lambda_2^{b N} B \Theta_1 - A^T \Lambda_{11}^{bN}  \\
\hspace{1.1cm}
- \Lambda_{11}^{bN} (A+G - B( \Theta + \Theta_1 ) )  + g_{11}^{bN} ,\\
%%%
\Lambda_{11}^{b N} (T) = 0 ,
\end{cases} \label{ODELam11dN}
\end{align}
\begin{align}
 & \begin{cases}
  \dot\Lambda_{12}^{b N} =   (\Lambda_2^{b N})^T B
 ( \mathcal{R}_1  (\Lambda_1^{b N} ) )^{-1} B^T \Lambda_{11}^{b N}
 - G^T ( \Lambda_{11}^{b N} + \Lambda_{12}^{b N} )   \\
 \hspace{1.1cm}
 - ( A^T  - \Theta^T B^T ) \Lambda_{12}^{b N}
 - \Lambda_{12}^{b N} (A+G - B(\Theta + \Theta_1 ) ) \\
 \hspace{1.1cm}
 - \Theta^T B_0^T
(\Lambda_1^{b N} + \Lambda_2^{b N} + (\Lambda_2^{bN})^T + \Lambda_4^{b N} ) B_0 \Theta_1  \\
 \hspace{1.1cm}
 - \Theta^T B_0^T ( \Lambda_{11}^{b N}  + \Lambda_{12}^{b N} ) B_0 (\Theta + \Theta_1 ) + \Lambda_4^{b N} B \Theta_1
 + g_{12}^{bN}  , \\
%%%
\Lambda_{12}^{b N}(T) = 0 ,
\end{cases}\label{ODELam12dN}
\end{align}
\begin{align}
 & \begin{cases}
   \dot{\Lambda}_{22}^{b N}
  =    (\Lambda_{11}^{b N})^T B
 ( \mathcal{R}_1( \Lambda_1^{b N} ) )^{-1} B^T \Lambda_{11}^{bN}
\\
\hspace{1.1cm}
  - \Lambda_{22}^{b N} (A+G - B ( \Theta + \Theta_1 ) )
  + (\Lambda_{12}^{b N})^T B \Theta_1  \\
\hspace{1.1cm}
  -  (A+G - B( \Theta + \Theta_1 ) )^T \Lambda_{22}^{b N}
 + \Theta_1^T B^T \Lambda_{12}^{b N}  \\
\hspace{1.1cm}
  - \Theta_1^T B_0^T ( \Lambda_1^{b N} + \Lambda_2^{b N}
 + (\Lambda_2^{bN})^T + \Lambda_4^{b N} ) B_0 \Theta_1 \\
\hspace{1.1cm}
 - ( \Theta + \Theta_1 )^T B_0^T \Lambda_{22}^{b N} B_0
 ( \Theta + \Theta_1 )   \\
\hspace{1.1cm}
- \Theta_1^T B_0^T ( \Lambda_{11}^{b N} + \Lambda_{12}^{b N} ) B_0 (\Theta + \Theta_1 ) \\
\hspace{1.1cm}
 - (\Theta + \Theta_1 )^T B_0^T
(  \Lambda_{11}^{b N}  + \Lambda_{12}^{b N} )^T B_0 \Theta_1
 + g_{22}^{bN}   ,  \\
%%%
\Lambda_{22}^{b N}(T) = 0 .
\end{cases} \label{ODELam22dN}
\end{align}
The  terms $g^{bN}_k$, $1\leq k \leq 4$, $g^{bN}_{11}$, $g^{bN}_{12}$ and $g^{bN}_{22}$, as functions of $(N, \Lambda_1^{bN}, \Lambda_2^{bN},\cdots , \Lambda_{22}^{bN})$, are given in Appendix~\ref{appendix:gdN}.

Let $(\Lambda_1^b, \Lambda_2^b,\cdots, \Lambda_{22}^b )$ be determined by the ODE system
\eqref{ODELam1d}--\eqref{ODELam22d} in Appendix ~\ref{appendix:PfsolLamdN}.
\begin{lemma} \label{lm:LambLam124}
Under Assumption \ref{assm:solLam}
we have
\begin{align}
&\Lambda_1^b(t)=\Lambda_1(t), \label{LbL1} \\
&\Lambda_2^b(t)+\Lambda_{11}^b(t)= \Lambda_2(t), \label{LbL2}\\
&\zeta^b(t)=  \Lambda_{2}(t)+  \Lambda_{2}^{T}(t)+  \Lambda_{4}(t)\label{LbL3}
\end{align}
for all $t\in [0,T]$,
where
$$
\zeta^b(t):=  (\Lambda_{2}^b+  \Lambda_{2}^{bT}+  \Lambda_{4}^b+  \Lambda_{11}^b+  \Lambda_{11}^{bT}+\Lambda_{12}^b+ \Lambda_{12}^{bT}+  \Lambda_{22}^b)(t).
$$
\end{lemma}

\begin{proof}
\eqref{LbL1} is already stated in the proof of  Lemma \ref{lemma:Lam1d}.
By considering the ODE of $\Lambda_2^b + \Lambda_{11}^b - \Lambda_2$
and next applying Gr\"{o}wnwall's lemma, we establish
$\sup_{0\le t\le T}|\Lambda_2^b(t) + \Lambda_{11}^b(t) - \Lambda_2(t)|=0$, which implies \eqref{LbL2}.

Define $\zeta(t)=\Lambda_{2}(t)+  \Lambda_{2}^{T}(t)+  \Lambda_{4}(t) $. By use of \eqref{ODELam2}, \eqref{ODELam4}, and
 \eqref{ODELam2d}--\eqref{ODELam22d} we write the ODEs:
\begin{align*}
&\dot{\zeta} (t) = \Phi( \Lambda_2, \Lambda_4),\\
&\dot{\zeta}^b(t)= \Phi^b(\Lambda_2^b, \Lambda_4^b, \Lambda^b_{11}, \Lambda_{12}^b, \Lambda_{22}^b),
\end{align*}
where the two vector fields are not fully displayed  but can be  easily determined.
Note that $\Lambda_1(t)$ and $(\Lambda_1(t), \Lambda_2(t))$ appear in $\Phi$ and $\Phi^b$, respectively, and are treated as known functions of time.  Letting $H=(\mathcal{R}_1(\Lambda_1))^{-1}$,  we have
\begin{align}
\Phi-\Phi^b =& (\zeta^b-\zeta)(A+G) +(A+G)^T (\zeta^b-\zeta)\\
 &+ (\Theta+\Theta_1)^T {B}_0^T (\zeta^b-\zeta){B}_0(\Theta+\Theta_1)\nonumber\\
 &- \Lambda_1 BHB^T (\zeta^b-\zeta) -(\zeta^b-\zeta) BHB^T \Lambda_1\nonumber\\
 &- (\zeta^b-\zeta) BHB^T \Lambda_2 + \Delta_\Phi, \nonumber
\end{align}
where we have used \eqref{LbL1}--\eqref{LbL2} to
derive the last line to get
\begin{align}
\Delta_\Phi = & \Lambda_2^T BHB^T \Lambda_2^T +  \Lambda_2^T BHB^T \Lambda_4   \nonumber\\ &
  - \Theta_1^T B^T  (\Lambda^{bT}_2   +\Lambda^{bT}_{11} + \Lambda^{bT}_{12} +\Lambda^b_4  +\Lambda^{bT}_{22} +\Lambda^b_{12}). \nonumber
\end{align}
By use of the definition of  $\zeta$ and \eqref{LbL2}, we obtain
 %\begin{align}
$\Delta_\Phi= \Lambda_2^T BHB^T (\zeta-\zeta^b)$. %\nonumber
 %\end{align}
By the ODE of $\zeta - \zeta^b$ and Gr\"{o}nwall's lemma, we obtain
$\sup_{t\in [0,T]}|\zeta(t) - \zeta^b(t)|=0$.
\end{proof}

Although the system \eqref{ODELam1dN}--\eqref{ODELam22dN}
has been constructed based on \eqref{ODEP1d}--\eqref{ODEP2d}, it can stand alone for its existence analysis without using the latter.
%%%%%%%%%%%
\begin{lemma}
\label{lm:solLamdN}
Under Assumption~\ref{assm:solLam}, there exists $N_1>0$ such that for all $N\geq N_1$,
\eqref{ODELam1dN}--\eqref{ODELam22dN} admits a solution $(\Lambda_1^{bN}, \cdots, \Lambda_{22}^{bN})$ on $[0, T]$ satisfying
\begin{align}\label{R1B0e0}
(\mathcal{R}_1(\Lambda_1^{bN} ) + B_0^T S^{bN} B_0/N^2)(t)>\epsilon_0 I , \quad \forall  t\in [0,T],
\end{align}
for some small constant $\epsilon_0>0$. In addition, $\sup_{t\in[0, T]} |\Lambda_\iota^{bN} - \Lambda_\iota^b|=O(1/N)$ for $\iota=1, 2, \cdots, 22$, where $\Lambda_1^b$, $\cdots$, $\Lambda_{22}^b$ are given in Appendix~\ref{appendix:PfsolLamdN}.
\end{lemma}
\begin{proof}
We  view  \eqref{ODELam1dN}--\eqref{ODELam22dN} as a slightly perturbed version of  \eqref{ODELam1d}--\eqref{ODELam22d}.  By the same thin tube method as in the sufficiency proof of Theorem~\ref{thm:NSAS},
we establish the existence and uniqueness of  a solution of
\eqref{ODELam1dN}--\eqref{ODELam22dN} for all sufficiently large $N$.
We may ensure \eqref{R1B0e0} due to $\mathcal{R}_1(\Lambda_1^b) >0$ for all $t\in [0,T]$ and a continuity argument.
 The error bound of $O(1/N)$ is obtained by applying Gr\"{o}nwall's lemma as in Corollary \ref{cor:NSAS:LamN-Lam}. \end{proof}

\begin{remark}
\label{rmk:bdLamdN}
Let $\psi^N$ stand for any of the functions $\Lambda_1^{bN}$,
$\Lambda_2^{bN}$, $\Lambda_3^{bN}$, $\Lambda_4^{bN}$,
$\Lambda_{11}^{bN}$, $\Lambda_{12}^{bN}$ and
$\Lambda_{22}^{bN}$.
Then $\sup_{N\geq N_1, 0 \leq t \leq T} |\psi^N| \leq C$ for some fixed constant $C$.
\end{remark}

\begin{remark}
\label{rmk:gdNO(1/N)}
Let $h^N$ stand for any of the functions $g_1^{bN}$, $g_2^{bN}$,
$g_3^{bN}$, $g_4^{bN}$, $g_{11}^{bN}$, $g_{12}^{bN}$ and
$g_{22}^{bN}$. Then $\sup_{t\in [0, T]} |h^N(t)|=O(1/N)$.
\end{remark}

\begin{lemma} \label{lm:exissolP1P2d}
Under Assumption \ref{assm:solLam},
the ODE system  \eqref{ODEP1d}--\eqref{ODEP2d} has a solution on $[0,T]$ for all $N\ge N_1$, where $N_1$ is specified in Lemma \ref{lm:solLamdN}. \end{lemma}

\begin{proof}
After obtaining $(\Lambda_1^{bN}, \cdots, \Lambda_{22}^{bN})$ by Lemma \ref{lm:solLamdN}, we define $\mathbf{P}_1^b$ using \eqref{P1dsubmat} and \eqref{PdLam}.
Then we can directly verify that $\mathbf{P}_1^b$ satisfies \eqref{ODEP1d}, where $ R + \mathbf{B}_1^T \mathbf{P}_1^b(t) \mathbf{B}_1 + \mathbf{B}_0^T \mathbf{P}_1^b(t) \mathbf{B}_0 >0 $  holds for all $t\in [0,T]$ since this matrix is equal to the term $  \mathcal{R}_1(\Lambda_1^{bN} ) + B_0^T S^{bN} B_0/N^2 $ appearing in  \eqref{ODELam1dN}. Note that \eqref{R1B0e0} holds.
Then we further uniquely solve \eqref{ODEP12d}--\eqref{ODEP2d}.
\end{proof}

Combining  Lemma
\ref{lm:exissolP1P2d} with Proposition \ref{prop:br} and Lemma \ref{lm:Pdsubmat}, we have the following facts. Under Assumption \ref{assm:solLam}, for all sufficiently large $N$,  the best response control problem  for player $\mathcal{A}_1$ has a solution. Next, the value function of the best response control problem can be specified using
\eqref{ODELam1dN}--\eqref{ODELam22dN}, which has a well defined solution.

%%%%%%%%%%
\begin{lemma}
\label{lm:Lam1dN-Lam1}
$\sup_{t\in[0, T]}|\Lambda_1^{bN}(t) - \Lambda_1(t)| = O(1/N)$.
\end{lemma}
\begin{proof} The lemma follows from Lemma \ref{lm:solLamdN} and \eqref{LbL1}.
\end{proof}

%%%%%%%%%%%%%
\begin{lemma}
\label{lm:checkLam1-Lam1}
$\sup_{t\in[0, T]} | \check\Lambda^N_1(t) - \Lambda_1(t)| = O(1/N)$.
\end{lemma}
\begin{proof}
Taking the difference of \eqref{ODEcheckLam1} and \eqref{ODELam1} gives
\begin{align}
\begin{cases}
\frac{d}{dt} (\check\Lambda_1^N - \Lambda_1)
= - \Theta^T B_1^T (\check\Lambda_1^N - \Lambda_1) B_1\Theta
- (\check\Lambda_1^N - \Lambda_1) (A-B\Theta)  \\
\hspace{2.4cm}
 - (A-B\Theta)^T (\check\Lambda_1^N - \Lambda_1) - \check{g}_1^N , \\
%%%
\check\Lambda_1^N(T) - \Lambda_1(T) =  (I - \Gamma_f^T/N) Q_f (I - \Gamma_f/N) - Q_f .
\end{cases} \notag
\end{align}
By Remark \ref{rmk:checkgO(1/N)}, $\sup_{t\in [0, T]} |\check{g}_1^N(t)|=O(1/N)$. The desired result follows from Gr\"{o}nwall's lemma.
\end{proof}

%%%%%%%%%%%%%
\begin{lemma}
\label{lm:Lam211dN-Lam2}
$\sup_{t\in[0, T]} |\Lambda_2^{bN}(t) + \Lambda_{11}^{bN}(t) - \Lambda_2(t) | = O(1/N)$.
\end{lemma}
\begin{proof} The lemma follows from Lemma \ref{lm:solLamdN} and \eqref{LbL2}.
\end{proof}

%%%%%%%%%%%%
\begin{lemma}
\label{lm:checkYN-YdN}
Let $\check{Y}^N$ be defined by \eqref{checkYN}, and denote
\begin{align}
 & Y^{b N} := \Lambda_1^{b N} + \Lambda_2^{b N}
 + (\Lambda_2^{b N})^T + \Lambda_4^{b N}
 + \Lambda_{11}^{b N} + ( \Lambda_{11}^{b N} )^T
 + \Lambda_{12}^{b N} + (\Lambda_{12}^{bN})^T
+ \Lambda_{22}^{b N}  . \notag
\end{align}
Then $\sup_{t\in[0, T]} |\check{Y}^N(t) - Y^{bN}(t)|=O(1/N)$.
\end{lemma}
\begin{proof}
Combining the ODEs \eqref{ODEcheckLam1}--\eqref{ODEcheckLam22} and
\eqref{ODELam1dN}--\eqref{ODELam22dN},
we obtain the following ODE of $\check{Y}^N - Y^{bN}$:
\begin{align*}
&\frac{d}{dt} ( \check{Y}^N - Y^{b N} )\\
&=  ( Y^{b N} - \check{Y}^N  ) (A+G -  B ( \Theta + \Theta_1 ) )
 + (A+G - B(\Theta + \Theta_1 ) )^T ( Y^{b N} - \check{Y}^N ) \notag \\
&
\quad+ (\Theta + \Theta_1 )^T
[ B_1^T (\Lambda_1^{b N} - \Lambda_1 ) B_1
 + B_0^T ( Y^{b N} - \check{Y}^N ) B_0 ]   ( \Theta + \Theta_1 )
 \notag \\
%%%%%%%%%%%
& \quad- ( \Lambda_1^{bN} + \Lambda_2^{bN}  + \Lambda_{11}^{bN} - \Lambda_1 - \Lambda_2 )^T B ( \mathcal{R}_1( \Lambda_1 ) )^{-1} B^T
 ( \Lambda_1^{bN} + \Lambda_2^{bN}  + \Lambda_{11}^{bN} - \Lambda_1 - \Lambda_2 )   \notag \\
&\quad  - (\Lambda_1 + \Lambda_2 )^T B ( \mathcal{R}_1 ( \Lambda_1 ) )^{-1} B_1^T (\Lambda_1^{bN} - \Lambda_1 ) B_1
  ( \mathcal{R}_1( \Lambda_1 ) )^{-1} B^T (\Lambda_1 + \Lambda_2 ) \notag \\
&\quad - (\Lambda_1^{bN} + \Lambda_2^{bN} + \Lambda_{11}^{bN} )^T B
 ( \mathcal{R}_1 ( \Lambda_1 ) )^{-1} B_1^T (\Lambda_1 - \Lambda_1^{bN} ) B_1 \cdot \\
&\qquad ( \mathcal{R}_1 (\Lambda_1^{bN} ) )^{-1} B^T
(\Lambda_1^{bN} + \Lambda_2^{bN} + \Lambda_{11}^{bN} )
+ \rho^N , \\
%%%%%%%%%%
&\check{Y}^N(T) - Y^{b N}(T) = 0 ,
\end{align*}
where
\begin{align}
\rho^N := & -(\check{g}_1^N + \check{g}_2^N + \check{g}_2^{NT}
+ \check{g}_4^N + \check{g}_{11}^N + \check{g}_{11}^{NT}
+ \check{g}_{12}^N + \check{g}_{12}^{NT} + \check{g}_{22}^N)
\notag \\
& - [g_1^{bN} + g_2^{bN} + (g_2^{bN})^T + g_4^{bN} + g_{11}^{bN} + (g_{11}^{bN})^T + g_{12}^{bN} + (g_{12}^{bN})^T + g_{22}^{bN}] .
\notag
\end{align}

The coefficients  of the term $\check{Y}^N - Y^{bN}$ are bounded.
By Lemmas \ref{lm:Lam1dN-Lam1} and \ref{lm:Lam211dN-Lam2},
we have $\sup_{t\in[0, T]}|\Lambda_1^{bN}(t) - \Lambda_1(t)| = O(1/N)$
and
$\sup_{t\in[0, T]}| \Lambda_1^{bN} + \Lambda_2^{bN}  + \Lambda_{11}^{bN} - \Lambda_1 - \Lambda_2  | = O(1/N)$.
By Remarks \ref{rmk:checkgO(1/N)} and \ref{rmk:gdNO(1/N)}, we have that
$\sup_{t\in[0, T]} |\rho^N| = O(1/N)$.
The lemma is then proven by applying Gr\"{o}nwall's lemma to the integral form of the ODE of $\check{Y}^N - Y^{bN}$.
\end{proof}

\begin{proof}[Proof of Theorem~\ref{thm:O(1/N)Nash}]

When all other players $\mathcal{A}_i$, $2\leq i \leq N$, take the decentralized strategies $\check{u}_{-1} = (\check{u}_2, \cdots, \check{u}_N)$, we compare the cost of player $\mathcal{A}_1$ under $u_1^b$ with the cost under $\check{u}_1$.
The cost $J_1(u_1^b, \check{u}_{-1})$ of $\mathcal{A}_1$ is
\begin{align}
  J_1(u_1^b, \check{u}_{-1})
=& \mathbb{E}\Big[   V_1^b(0, X(0), \overline{X}(0))   \Big]   \notag \\
= & \mathbb{E} \Big[ X^T(0) \mathbf{P}_1^b(0)  X(0)
+ 2 X^T(0) \mathbf{P}_{12}^b(0)  \overline{X}(0)
 + \overline{X}^T(0)  \mathbf{P}_2^b(0)  \overline{X}(0) \Big]  ,
 \notag \\
=& \mu_0^T   Y^{bN}(0)  \mu_0
+  \Tr[  \Lambda_1^{bN}(0) \Sigma_0^1 ]
 + (1/N^2) \sum_{ i = 2}^N \Tr [ \Lambda_3^{bN}(0) \Sigma_0^i]    \label{J1Pd} \\
& + \mu_0^T \{  - (\Lambda_2^{bN}(0) + (\Lambda_2^{bN}(0))^T) /N +  \Lambda_3^{bN}(0) (N-1)/N^2 \notag \\
 & +  \Lambda_4^{bN}(0)  (2-3N)/N^2
  - (   \Lambda_{12}^{bN}(0)  + (\Lambda_{12}^{bN}(0))^T  ) /N
\}  \mu_0 .   \notag
\end{align}
The cost $J_1(\check{u}_1, \check{u}_{-1})$ can be obtained from \eqref{checkJiP}.
Then we have
\begin{align}\label{diff:Jd-Jcheck}
  J_1(u_1^b, \check{u}_{-1}) - J_1(\check{u}_1, \check{u}_{-1})
=& \mu_0^T [ Y^{bN}(0) - \check{Y}^N(0) ] \mu_0
+  \Tr[ (\Lambda_1^{bN}(0) - \check\Lambda_1^N(0)) \Sigma_0^1 ]
  \\
& + (1/N^2) \sum_{i=2}^N \Tr [(\Lambda_3^{bN}(0) - \check\Lambda_3^N(0)) \Sigma_0^i] + O(1/N), \notag
\end{align}
where we obtain the estimate $O(1/N)$ using Remarks \ref{rmk:bdcheckLam} and \ref{rmk:bdLamdN}.
By Lemma \ref{lm:checkYN-YdN}, we have
\begin{align}
|\mu_0^T [ Y^{bN}(0) - \check{Y}^N(0) ] \mu_0| = O(1/N).
\label{estmu0Y}
\end{align}
From Lemmas \ref{lm:Lam1dN-Lam1} and \ref{lm:checkLam1-Lam1}, we have $\sup_{t\in[0, T]} |\Lambda_1^{bN}(0) - \check\Lambda_1^N(0)| = O(1/N)$ and thus
\begin{align}
 |\Tr[ (\Lambda_1^{bN}(0) - \check\Lambda_1^N(0)) \Sigma_0^1 ] | =O(1/N) . \label{estLam1Sigma}
\end{align}
By Assumption \ref{assm:initialX} and Remarks \ref{rmk:bdcheckLam} and \ref{rmk:bdLamdN}, we have
\begin{align}
& (1/N^2)\Big| \sum_{i=2}^N \Tr [(\Lambda_3^{bN}(0) - \check\Lambda_3^N(0)) \Sigma_0^i]\Big|     = O(1/N) . \label{estLam3Sigma}
\end{align}
It follows from \eqref{diff:Jd-Jcheck} and \eqref{estmu0Y}, \eqref{estLam1Sigma} and  \eqref{estLam3Sigma}  that
\begin{align}
 0\le J_1(\check{u}_1, \check{u}_{-1})-J_1(u_1^b, \check{u}_{-1}) = O(1/N).
 \label{J1eNEN}
\end{align}
Note that the term $O(1/N) $ in \eqref{J1eNEN} does not depend on which player is selected to apply its best response.
This completes the proof.
\end{proof}

Let $u_i^b$ denote the best response strategy of ${\mathcal A}_i$ when all other players apply their strategies $\check u_{-i}$.

\begin{theorem}\label{thm:J1J2J3}
Under Assumptions \ref{assm:solLam} and \ref{assm:initialX}, we have
\begin{align}
&\max_{1\le i\le N}|J_i(u_i^b, {\check u}_{-i})- J_i(\hat u_i, \hat u_{-i})|
=O(1/N), \label{Jibhat}  \\
&\max_{1\le i\le N}  |J_i(\check{u}_i, {\check u}_{-i})- J_i(\hat u_i, \hat u_{-i})|=O(1/N). \label{Jichat}
\end{align}
\end{theorem}

\begin{proof}
By using the value function of the $N$-player Nash game, we have
\begin{align}
J_i(\hat u_i, \hat u_{-i}) =
& \mu_0^T   [ \Lambda_1^N(0) + \Lambda_2^N(0) + \Lambda_2^{NT}(0) + \Lambda_4^N(0) ]  \mu_0
+  \Tr[  \Lambda_1^N(0) \Sigma_0^i ] \label{Jihathat0}   \\
& + (1/N^2) \sum_{ i\neq j = 1}^N \Tr [ \Lambda_3^N(0) \Sigma_0^j]
+ \mu_0^T \{  - (\Lambda_2^N(0) + \Lambda_2^{NT}(0)) /N  \notag \\
 & +  \Lambda_3^N(0) (N-1)/N^2 +  \Lambda_4^N(0)  (2-3N)/N^2 \}  \mu_0
\notag  \\
=& \mu_0^T[\Lambda_1^N(0) +\Lambda_2^N(0)+\Lambda_2^{NT}(0)+\Lambda_4^N(0)]
\mu_0 \notag \\
& + {\rm Tr}[\Lambda_1^N(0)\Sigma_0^i] +O(1/N)\notag \\
=&  \mu_0^T[\Lambda_1(0) +\Lambda_2(0)+\Lambda_2^{T}(0)+\Lambda_4(0)]
\mu_0   \label{JiLam14mu} \\
 &
+ {\rm Tr}[\Lambda_1(0)\Sigma_0^i] +O(1/N),\nonumber
\end{align}
where the last equality follows from Corollary \ref{cor:NSAS:LamN-Lam}.
Similarly, we use \eqref{J1Pd} and Lemma \ref{lm:solLamdN} to obtain
\begin{align}
J_i(u_i^b, \check u_{-i})=&  \mu_0^TY^{bN}(0)
\mu_0 + {\rm Tr}[\Lambda_1^{bN}(0)\Sigma_0^i] +O(1/N) \nonumber   \\
=&  \mu_0^T[\Lambda_1(0)+ \zeta^b(0)]
\mu_0 +{\rm Tr}[\Lambda_1(0)\Sigma_0^i] +O(1/N). \label{JibLamzeta}
\end{align}
The term $O(1/N)$ in all estimates obtained above  does not depend on $i$.
By \eqref{JiLam14mu}--\eqref{JibLamzeta} and Lemma \ref{lm:LambLam124}, we obtain \eqref{Jibhat}, which combined with Theorem \ref{thm:O(1/N)Nash}  yields \eqref{Jichat}.
\end{proof}

\subsection{The general model}

Now we consider a general LQ  model where
$D$ and $D_0$ in \eqref{Xi} may be nonzero and where the cost \eqref{Ji} is modified by using the running cost
$ [X_i(t) - \Gamma X^{(N)}(t)-\eta]_Q^2 + [u_i(t)]_{R}^2$ and
the terminal cost
$ [X_i(T) - \Gamma_f X^{(N)}(T)-\eta_f]_{Q_f}^2$ for $\eta, \eta_f\in \mathbb{R}^n$. Then all the previous analysis in Sections \ref{sec:RiccatiEqns}--\ref{sec:O(1/N)Nash} may be easily adapted to this general model.

The value function in \eqref{Vansatz} is now replaced by the form
$$
V^G(t,\mathbf{x})=  \mathbf{x}^T \mathbf{P}_i(t) \mathbf{x} + 2 \mathbf{x}^T\mathbf{S}_i^{G}(t)  + \mathbf{r}_i^G (t), \quad
 1\leq i \leq N .
$$
The same ODE system \eqref{ODEP}--\eqref{Pcon} is used for $(\mathbf{P}_1, \cdots, \mathbf{P}_N)$. If $(\mathbf{P}_1, \cdots, \mathbf{P}_N)$ is given on $[0,T]$, then $ (\mathbf{S}_1^{G}, \cdots, \mathbf{S}_N^{G})$ is uniquely solved from a linear ODE system. Finally, given $(\mathbf{P}_i, \mathbf{S}^G_i)$, $1\le i\le N$, on $[0,T]$, each $\mathbf{r}^G_i$ is again solved from a linear ODE. For this general model,
Definition \ref{def:AS} about asymptotic solvability remains valid, and Theorem \ref{thm:NSAS} still holds.   The asymptotic analysis can be extended to treat  $\{\mathbf{S}_i^{G},  \mathbf{r}_i^{G}, 1\le i\le N \}$. We can accordingly determine  the Nash equilibrium strategies $\hat u_i^G$, $1\le i\le N$, the decentralized strategies $\check u_i^G$, $1\le i\le N$, and the best response strategy $u_i^{Gb}$ given $ \check u_{-i}^G$, which are further used to establish Theorems
\ref{thm:O(1/N)Nash} and \ref{thm:J1J2J3}.
 We summarize the following result:
\begin{corollary}
Under Assumptions \ref{assm:solLam} and \ref{assm:initialX}, Theorems \ref{thm:O(1/N)Nash} and \ref{thm:J1J2J3} still hold for the general model with parameters $(D,D_0, \eta, \eta_f)$.
\end{corollary}

%%%%%%%%%%%%%%%%%%%%%%%%%%%%%%

%%%%%%%%%%%%%%%%%%%%%%%%%%%
%%%%%%%%%%%%%%%%%%%%%%%%%%%
\section{Numerical example}
\label{sec:num}

We present a numerical example to illustrate asymptotic solvability and  individual costs.
 %the decentralized strategies $(\check u_1, \cdots, \check u_N)$.
The parameter values  are
$A =-1$, $B = 1$, $B_0=-2$, $B_1=4$, $G = 1$, $R=-1$, $Q =8$,
$\Gamma =0.8$,
$Q_f = 8$, $\Gamma_f = 0.8$, and $T=2$.
We take the initial conditions $X_i(0)=1$ for all $i\geq 1$, and so $\overline{X}(0)=1$.

When \eqref{ODELam1}--\eqref{ODELam2} admits a solution $(\Lambda_1, \Lambda_2)$ on $[0, T]$, we use MATLAB ODE solver \texttt{ode45} to solve \eqref{ODELam1}--\eqref{ODELam4} to obtain the solution
$(\Lambda_1, \Lambda_2, \Lambda_3, \Lambda_4)$.
At $t=0$, we obtain $ \Lambda_1(0)=3.9435$, $\Lambda_2(0)=-2.3751$, $\Lambda_3(0)=1.8351$  and $\Lambda_4(0)=1.7786$.
Fig. \ref{fig:Lam12J} (left panel) shows that
 \eqref{ODELam1}--\eqref{ODELam2} admits a solution $(\Lambda_1, \Lambda_2)$ on $[0, T]$ so that the Nash game \eqref{Xi}--\eqref{Ji}
 has asymptotic solvability. By the initial conditions and
 \eqref{JiLam14mu}, under Nash strategies the asymptotic per agent
  cost is  $\lim_{N\to\infty}J_i(\hat u_i, \hat u_{-i}) 
  = \Lambda_1(0)+2\Lambda_2(0)+\Lambda_4(0)=0.9719$, which is indicated by the dashed horizonal line in Fig. \ref{fig:Lam12J} (right panel).
Fig. \ref{fig:Lam12J} (right panel) shows that   as $N$ increases,
the cost $J_i(\check u_i, \check u_{-i})$ of player $\mathcal{A}_i$ under the set of decentralized strategies  approaches  $\lim_{N\to \infty}J_i(\hat{u}_i, \hat{u}_{-i})$, as asserted by Theorem \ref{thm:J1J2J3}.

%%%%%%%%%%
\begin{figure}[h]
\begin{centering}
\begin{tabular}{cc}
\hspace{-0.2cm}
\begin{minipage}{0.57\textwidth}
\includegraphics[width=1.2\textwidth, height=5cm]{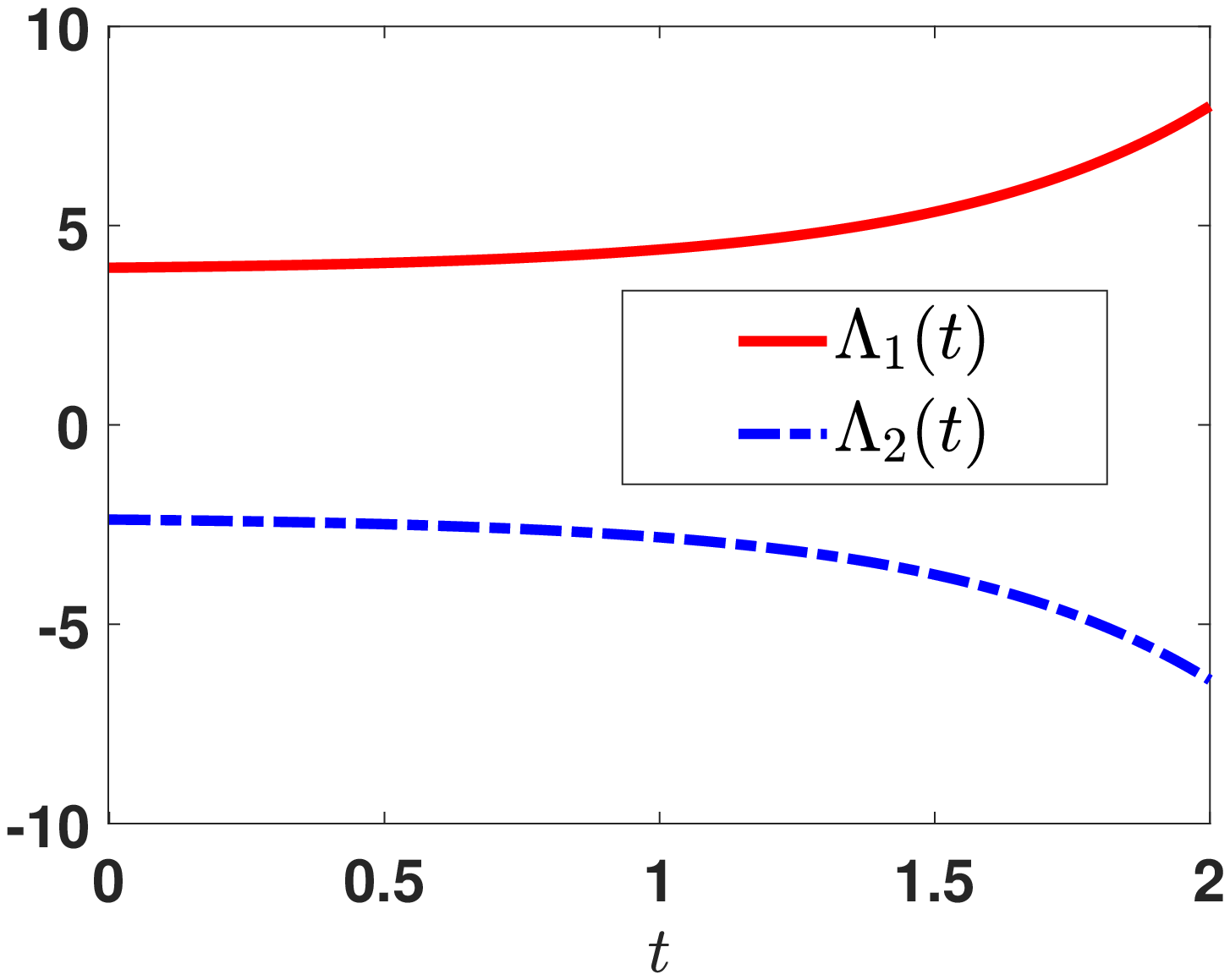}
%[width=0.98\textwidth, height=2.4in,trim=1in 2.65in 1in 2.2in]
\end{minipage} &
\hspace{-1cm}
\begin{minipage}{0.57\textwidth}
\includegraphics[width=1.2\textwidth, height=5cm]{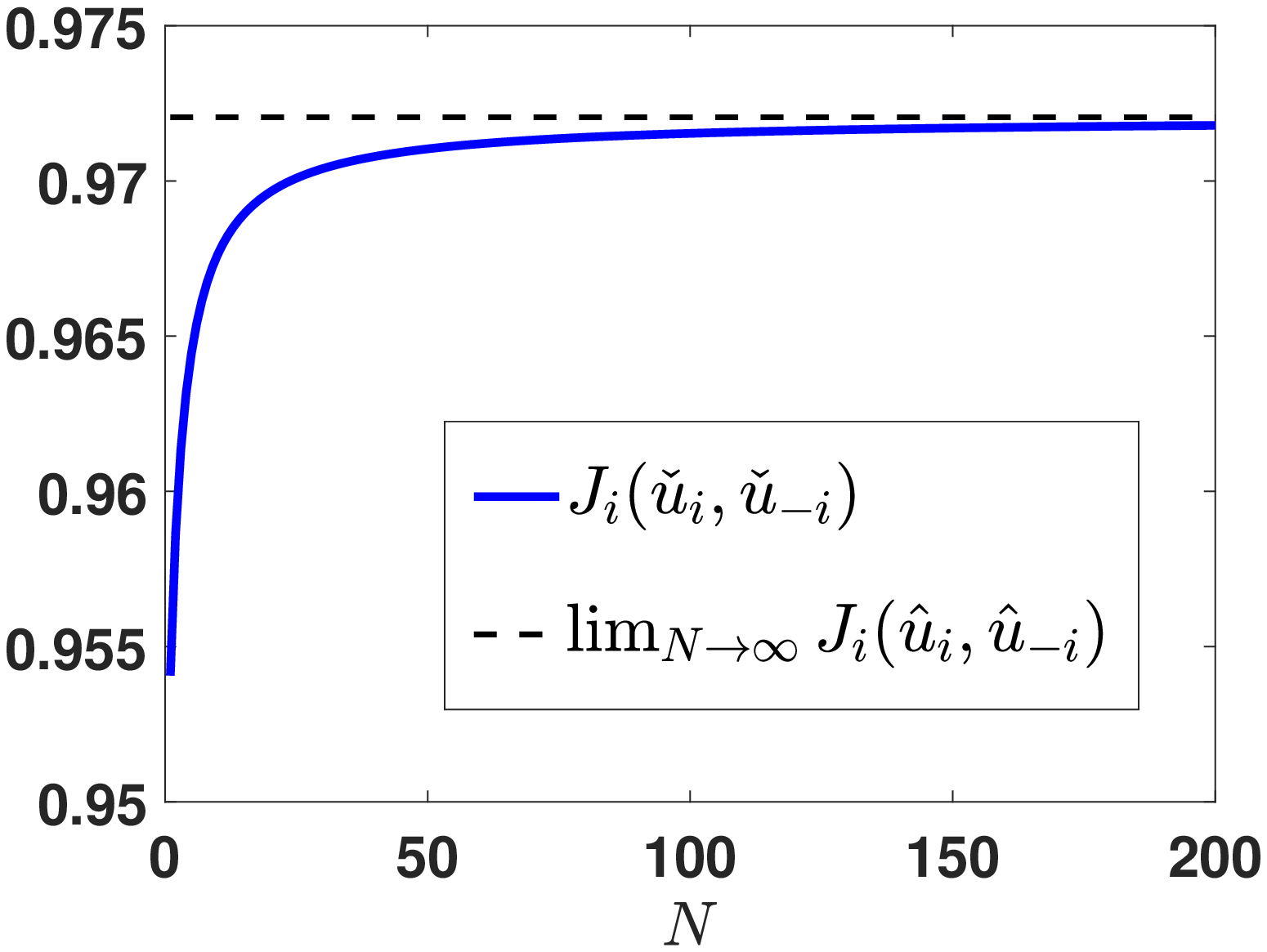}
\end{minipage}
\end{tabular}
\begin{minipage}{0.97\textwidth}
\caption{
Left panel: $(\Lambda_1, \Lambda_2)$
admits a solution on $[0, T]$ with $T=2$.
Right panel: The cost of player $\mathcal{A}_i$ under the set of decentralized strategies $(\check{u}_i, \check{u}_{-i})$ converges to a limit as $N\to \infty$.}
\label{fig:Lam12J}
\end{minipage}
\end{centering}
\end{figure}

%%%%%%%%%%%%%%%%%%%%
\section{Conclusion}
\label{sec:conclusion}

This paper studies an asymptotic solvability problem for LQ mean field games with controlled diffusions and  indefinite cost weights.
  By a rescaling approach
we derive a necessary and sufficient condition for asymptotic solvability.
 We further establish an $O(1/N)$-Nash equilibrium property for the obtained decentralized strategies.

%%%%%%%%%%%%%%
\appendix

\section{}
\label{appendix:pflmPsubmat}

\begin{proof}[Proof of Theorem \ref{thm: FBNE}] To show the feedback Nash equilibrium property, we let ${\mathcal A}_k$, $k=2, \cdots, N$, take the strategies in \eqref{uiP} and ${\mathcal A}_1$ unilaterally improves for itself. We need to show the optimality of $\hat u_1$ for minimizing $ J_1(t,\mathbf{x}, u_1, \hat u_{-1}) $ for any given $(t, \mathbf{x})$.

Step 1.
 Denote the Riccati ODEs \eqref{ODEP} in the form
\begin{align}
-\dot{\mathbf{P}}_i=\Phi_i(\mathbf{P}_1, \cdots,  \mathbf{P}_N), \quad
\mathbf{P}_i(T) = \mathbf{Q}_{i f},  \quad 1\le i\le N. \label{RicvfPhi}
\end{align}
Let $(\mathbf{P}_2, \cdots, \mathbf{P}_N)$ still be specified by \eqref{RicvfPhi}. We will derive a new but equivalent ODE for $\mathbf{P}_1$. It is necessary to do so since the best response control problem will give rise to a Riccati equation not exactly in the form of \eqref{ODEP} with $i=1$.
For parameter $\mathbf{x}\in \mathbb{R}^{Nn}$, based on \eqref{NEuhat} we consider the following equation system
\begin{align}
&0 = \widehat{\mathbf{B}}_i^T \mathbf{P}_i(t)\mathbf{x}
 + \mathbf{B}_0^T \mathbf{P}_i(t) \mathbf{B}_0 \sum_{k\neq i}^N  {u}_k^{\mathbf{x}} + [\mathbf{B}_0^T  \mathbf{P}_i(t) \mathbf{B}_0+ \mathbf{B}_i^T \mathbf{P}_i(t)\mathbf{B}_i   +  R]  u_i^{\mathbf{x}} , \label{equx} \\
  &1\le i\le N, \nonumber
\end{align}
which under \eqref{Pcon} has a unique solution
\begin{align}
 u_i^{\mathbf{x}} &= - [  R + \mathbf{B}_i^T  \mathbf{P}_i(t)
 \mathbf{B}_i  ]^{-1}
 [ \mathbf{B}_0^T \mathbf{P}_i(t) \mathbf{B}_0  \mathbf{M}_{0}(t)   +  \widehat{\mathbf{B}}_i^T  \mathbf{P}_i(t) ] \mathbf{x} \nonumber\\
 &=:{\mathbf K}_i(t)\mathbf{x}, \quad 1\le i\le N.
\end{align}

Now for $i=1$, we  further use \eqref{equx} to obtain
\begin{align}
u_1^{\mathbf{x}}=&-[ R+ \mathbf{B}_1^T \mathbf{P}_1(t)\mathbf{B}_1   +\mathbf{B}_0^T  \mathbf{P}_1(t) \mathbf{B}_0
]^{-1}\\
&\quad \cdot \Big\{ \widehat{\mathbf{B}}_1^T
 \mathbf{P}_1(t)\mathbf{x}
 + \mathbf{B}_0^T \mathbf{P}_1(t) \mathbf{B}_0 \sum_{k= 2}^N{\mathbf K}_k(t)\mathbf{x}  \Big  \} \nonumber \\
 &= : \widetilde{\mathbf{K}}_1 (t) \mathbf{x}. \nonumber
\end{align}
Since $\mathbf{x}$ is arbitrary, we obtain the identity
\begin{align}
{\mathbf K}_1(t) = \widetilde{\mathbf K}_1(t), \quad \forall t\in [0,T]. \end{align}

Although the Riccati equation system \eqref{ODEP} may be written down without using the HJB equation \eqref{HJBV}, the following observation is useful. For each $\mathbf{P}_i$, the vector field in \eqref{ODEP} may be constructed from the quadratic form determined by the right hand side of \eqref{HJBV}. For illustration, take $k\ne i$. Then $\partial_{\mathbf{x}}^TV_i \widehat{\mathbf{B}}_k
\hat u_k$ in \eqref{HJBV} contributes $ \mathbf{P}_i \widehat{\mathbf{B}}_k {\mathbf{K}}_k(t)+ {\mathbf{K}}_k^T(t) \widehat{\mathbf{B}}^T_k \mathbf{P}_i $ contained in the right hand side of \eqref{ODEP}.
Now for the ODE \eqref{ODEP} of $\mathbf{P}_1$, whenever a term originates from $\hat u_1$ so that ${\mathbf K}_1(t)$ is used in the vector field, we replace $\mathbf{K}_1(t)$ by $\widetilde{\mathbf K}_1(t)$.
For example, now
 $\mathbf{M}_{0}^T \mathbf{B}_0^T \mathbf{P}_1 \mathbf{B}_0 \mathbf{M}_{0}$
 is replaced by
 $$
 \Big(\widetilde{\mathbf K}_1(t)+\sum_{j=2}^N {\mathbf K}_j(t)\Big)^T \mathbf{B}_0^T \mathbf{P}_1 \mathbf{B}_0 \Big(\widetilde{\mathbf K}_1(t)+\sum_{j=2}^N {\mathbf K}_j(t)\Big),
 $$
which is ultimately expressed in terms of $(\mathbf{P}_1, \cdots,\mathbf{P}_N )$.
By the above substitution of ${\mathbf K}_1(t)$ by  $\widetilde{\mathbf K}_1(t)$, we see that $\mathbf{P}_1$ satisfies the new equation
\begin{align}
-\dot{\mathbf{P}}_1= \Phi_1^{\rm new}(\mathbf{P}_1, \cdots,  \mathbf{P}_N), \quad \mathbf{P}_1(T) = \mathbf{Q}_{1f}. \label{P1ODEnew}
\end{align}
The vector field $\Phi_1^{\rm new}$ is not fully displayed here to save space, but can be easily determined.
We note that the term $ \mathbf{P}_i \widehat{\mathbf{B}}_k {\mathbf{K}}_k(t)+ {\mathbf{K}}_k^T(t) \widehat{\mathbf{B}}^T_k \mathbf{P}_i $, $k\ne i$, mentioned above remains in $\Phi_1^{\rm new}$.
Then $(\mathbf{P}_1, \cdots,\mathbf{P}_N )$ is uniquely solved from
the ODE system specified by
$\Phi_1^{\rm new}, \Phi_2, \cdots, \Phi_N$.

 Step 2. Consider the initial time and state pair $(t,\mathbf{x})$, $t\in [0,T)$.  Suppose players ${\mathcal A}_k$, $k\ge 2$, apply the strategies in \eqref{uiP} while ${\mathcal A}_1$ minimizes the cost
$J_1(t,\mathbf{x}, u_1, \hat u_{-1})$ with the initial condition $(t,\mathbf{x})$. The resulting optimal control $u_1^{\rm br}$ on $[t, T]$ is  its best response. By considering the state process $X(s)$, $s\in [t,T]$ under $(u_1, \hat u_{-1})$ and the cost $J_1(t, \mathbf{x}, u_1, \hat u_{-1}) $,  it is straightforward to determine the Riccati equation of this optimal control problem in the form
\begin{align}
\dot{\mathbf{P}}_1^{\rm br}= \Phi^{\rm br} (\mathbf{P}_1^{\rm br};\mathbf{P}_1, \cdots,  \mathbf{P}_N ),\quad  \mathbf{P}_1^{\rm br}(T) = \mathbf{Q}_{1 f},  \quad s\in [t, T], \label{P1br}
\end{align}
where $ (\mathbf{P}_1, \cdots,  \mathbf{P}_N)$ specifies coefficients of \eqref{P1br} and has been solved from \eqref{P1ODEnew} and \eqref{RicvfPhi} with $i\ge 2 $.
By comparing the structure of \eqref{P1ODEnew} and \eqref{P1br}, we see that  \eqref{P1br} is verified by taking ${\mathbf{P}}_1^{\rm br}=\mathbf{P}_1$.
In particular, for the inverse term $[ R+  \mathbf{B}_1^T \mathbf{P}^{\rm br}_1(t)\mathbf{B}_1+ \mathbf{B}_0^T  \mathbf{P}^{\rm br}_1(t) \mathbf{B}_0  ]^{-1}$ appearing in \eqref{P1br},  we have
$$
   R+  \mathbf{B}_1^T \mathbf{P}^{\rm br}_1(t)\mathbf{B}_1+ \mathbf{B}_0^T  \mathbf{P}^{\rm br}_1(t) \mathbf{B}_0>0
$$
since $ R+  \mathbf{B}_1^T \mathbf{P}_1(t)\mathbf{B}_1+ \mathbf{B}_0^T
\mathbf{P}_1(t) \mathbf{B}_0 >0$ holds and we have taken $\mathbf{P}_1^{\rm br}=\mathbf{P}_1$.
 By uniqueness, we see that $\mathbf{P}_1^{\rm br}$ must be equal to $\mathbf{P}_1$ on $[t,T]$. The best response is well defined on $[t,T]$, and we use
\eqref{P1br} and $\mathbf{P}_1^{\rm br}$ to determine
 \begin{align}
u_1^{\rm br}(s)=\widetilde{\mathbf{K}}_1(s)X(s) ={\mathbf{K}}_1(s)X(s) , \quad s\in [t,T].
\end{align}
The optimality of $u_1^{\rm br}$ may be shown by using \eqref{P1br} and applying  completion of squares to the cost (see \cite[Theorem 6.6.1]{YZ1999}).
 Hence $\hat u_1$ gives the best response for $\mathcal{A}_i$ on $[t,T]$.

Step 3. The same best response property holds for $\hat u_i$ when any other single player $\mathcal{A}_i$ is chosen for unilateral performance improvement.
 We conclude that the feedback Nash equilibrium property holds.
\end{proof}

\begin{proof}[Proof of Lemma \ref{lm:Psubmat}]
The proof is carried out in the same manner as that of \cite[Theorem 3]{HZ2020}.

Step 1. For each $2\leq j < l \leq N$, denote $\mathbf{P}_i^\dagger = J_{jl}^T \mathbf{P}_i J_{jl}$, $1\leq i \leq N$.
We have that
\begin{align}
( \mathbf{P}_1^{\dagger}, \cdots, \mathbf{P}_{j-1}^{\dagger}, \
\mathbf{P}_l^{\dagger}, \ \mathbf{P}_{j+1}^{\dagger}, \ \cdots, \
\mathbf{P}_{l-1}^{\dagger} , \ \mathbf{P}_j^{\dagger} , \
\mathbf{P}_{l+1}^{\dagger} , \ \cdots, \ \mathbf{P}_N^{\dagger} )
\notag
\end{align}
satisfies the same ODE system \eqref{ODEP}--\eqref{Pcon} as
$(\mathbf{P}_1, \cdots , \mathbf{P}_N)$ does.
Thus $J_{jl}^T \mathbf{P}_1 J_{jl} = \mathbf{P}_1$ for all $2\leq j < l \leq N$.
Denote $\mathbf{P}_i = ( P^i_{jl})_{1\leq  j, l \leq N}$, where each $P^i_{jl}$ is an $n\times n$ matrix.
Then we have that
\begin{align}
 & P^1_{12} = P^1_{13} = \cdots = P^1_{1N} , \quad P^1_{21} = P^1_{31} = \cdots = P^1_{N1} , \notag \\
& P^1_{22} = P^1_{33} = \cdots  =P^1_{NN} , \quad
P^1_{j_1 l_1} = P^1_{j_2 l_2} , \quad \forall 2\leq j_1 \neq l_1 \leq N, \ \forall 2\leq j_2 \neq l_2 \leq N .  \notag
\end{align}
This proves the representation of $\mathbf{P}_1$ in \eqref{Psubmat}.

Step 2. For each $2\leq j \leq N$, denote $\mathbf{P}_i^\ddagger = J_{1j}^T \mathbf{P}_i J_{1j}$, $1\leq i \leq N$. We have that
\begin{align}
( \mathbf{P}^\ddagger_j, \ \mathbf{P}^\ddagger_2, \cdots, \mathbf{P}^\ddagger_{j-1}, \ \mathbf{P}^\ddagger_1, \ \mathbf{P}^\ddagger_{j+1},\ \cdots, \ \mathbf{P}^\ddagger_N )   \notag
\end{align}
and $(\mathbf{P}_1, \cdots , \mathbf{P}_N)$ both
satisfy  \eqref{ODEP}--\eqref{Pcon}.
This implies $\mathbf{P}_j = J_{1j}^T \mathbf{P}_1 J_{1j}$.
\end{proof}

%%%%%%%%%%
\section{ }
\label{appendix:gdN}

\subsection{Perturbation terms used in \eqref{ODELam2N}--\eqref{ODELam4N}}
\begin{align}
  g_2^N =
 &  - \Lambda_2^N B [ K^N B^T S_{12}^N  (N-1)/N
 + 2 H^N B^T  \Lambda_2^N  ] / N    \notag \\
 &   + [ \Lambda_2^{NT} B H^N
 - S_{12}^{NT} B K^{NT}/N ] B^T S_{34}^N
 \notag \\
  & - S_{12}^{NT} B K^{NT} B_1^T \Lambda_3^N B_1
K^N B^T S_{12}^N (N-1)/N^5
 \notag  \\
 & -  \Lambda_2^{NT} B H^N B_1^T \Lambda_3^N B_1 H^N B^T
 [ \Lambda_1^N + \Lambda_2^N (N-2) /N  ] / N^2  \notag \\
 &  - S_{12}^{NT} B K^{NT} B_1^T
 \Lambda_3^N B_1 H^N B^T [ - S_{12}^N/N^3 +  \Lambda_2^N / N^4 ]
 \notag  \\
  & + \Lambda_2^{NT} B H^N B_1^T \Lambda_3^N B_1 K^N B^T
S_{12}^N (N-1) / N^4   \notag \\
& +[ - S_{12}^{NT} B F^N B^T  S_{12}^N     +   \Lambda_2^N  G
  - G^T \Lambda_2^N
 -  \Gamma^T Q \Gamma ] /N
 - G^T S_{34}^N ,  \notag
\end{align}
\begin{align}
 g_3^N =&  - S_{34}^N B K^N B^T  S_{12}^N
 - S_{12}^{NT} B K^{NT} B^T S_{34}^N
  \notag \\
 &   - 2 [ \Lambda_4^N B H^N B^T \Lambda_2^N
  + \Lambda_2^{NT} B H^N B^T \Lambda_4^N ] / N
  \notag  \\
 & - S_{12}^{NT} B K^{NT} B_1^T \Lambda_3^N B_1
K^N B^T S_{12}^N (N-1)/N^4  \notag \\
&   - \Lambda_2^{NT} B H^N B_1^T \Lambda_3^N B_1 H^N B^T \Lambda_2^N (N-2)/N^2  \notag \\
&  + S_{12}^{NT} B K^{NT} B_1^T
 \Lambda_3^N B_1 H^N B^T  [ S_{12}^N  - \Lambda_2^N / N ]
 /N^2
\notag \\
&  + [ S_{12}^{NT}  - \Lambda_2^{NT}/ N ] B H^N
  B_1^T \Lambda_3^N B_1 K^N B^T  S_{12}^N / N^2    \notag \\
  & - [ (  \Lambda_3^N - 2 \Lambda_4^N  ) G
    + G^T (  \Lambda_3^N   -2 \Lambda_4^N  ) ] / N
 -  S_{12}^{NT} B F^N B^T S_{12}^N   \notag \\
 &  + [ \Lambda_1^N +  \Lambda_2^{NT} ] B H^N B_0^T ( \Lambda_1^N + \Lambda_2^N+\Lambda_2^{NT} + \Lambda_4^N ) B_0 H^N B^T
 [ \Lambda_1^N +  \Lambda_2^N  ] ,  \notag
\end{align}
\begin{align}
g_4^N = & - S_{34}^N B K^N B^T  S_{12}^N
 - S_{12}^{NT} B K^{NT} B^T  S_{34}^N   \notag \\
 & +[ ( \Lambda_3^N - 3 \Lambda_4^N ) B H^N B^T \Lambda_2^N
 + \Lambda_2^{NT} B H^N B^T ( \Lambda_3^N - 3 \Lambda_4^N) ] /N
   \notag \\
  & - S_{12}^{NT} B K^{NT} B_1^T \Lambda_3^N B_1
K^N B^T S_{12}^N (N-1)/N^4  \notag \\
& - [ \Lambda_1^N B H^N B_1^T \Lambda_3^N B_1 H^N B^T \Lambda_2^N
  + \Lambda_2^{NT} B H^N B_1^T \Lambda_3^N B_1 H^N B^T \Lambda_1^N ] / N   \notag \\
&   - \Lambda_2^{NT} B H^N B_1^T \Lambda_3^N B_1 H^N B^T \Lambda_2^N (N-3) / N^2    \notag \\
&  + S_{12}^{NT} B K^{NT} B_1^T
 \Lambda_3^N B_1 H^N B^T  [ S_{12}^N  - \Lambda_2^N/N ] /N^2 \notag \\
 & + [ S_{12}^{NT}  - \Lambda_2^{NT} / N ] B H^N
  B_1^T \Lambda_3^N B_1 K^N B^T  S_{12}^N / N^2
  \notag\\
   & +  [ \Lambda_1^N +  \Lambda_2^{NT} ] B H^N B_0^T ( \Lambda_1^N + \Lambda_2^N+\Lambda_2^{NT} + \Lambda_4^N ) B_0 H^N B^T
 [ \Lambda_1^N +  \Lambda_2^N  ]   \notag \\
& - S_{12}^{NT} B F^N B^T S_{12}^N
 - [ ( \Lambda_3^N  - 2\Lambda_4^N ) G
   + G^T ( \Lambda_3^N  -2  \Lambda_4^N  ) ] /N   , \notag
\end{align}

\subsection{Perturbation terms used in \eqref{ODEcheckLam1}--\eqref{ODEcheckLam22}}
\begin{align}
 \check{g}_1^N := &  ( \check{\Lambda}_1^N G + G^T \check{\Lambda}_1^N  ) /N
 +  ( \check{\Lambda}_2^N G + G^T \check{\Lambda}_2^{NT} ) (N-1)/N^2
\notag \\
& +   ( \Gamma^T Q \Gamma/N  - \Gamma^T Q - Q \Gamma )/N
 +  \Theta^T B_0^T\check{S}^N B_0 \Theta / N^2 , \notag \\
 \check{g}_2^N
  := &   ( \Gamma^T Q \Gamma -\check{\Lambda}_2^N G )/N
+ G^T [ \check{\Lambda}_2^N/N + \check{\Lambda}_3^N /N^2 + \check{\Lambda}_4^N (N-2)/N^2 ]   \notag \\
& + \Theta^T B_0^T \check{S}^N  B_0 \Theta /N , \notag\\
\check{g}_3^N  =\ &\check{g}_4^N \notag \\
:= & (1/N) [ ( \check\Lambda_3^N-2 \check\Lambda_4^N )G
 + G^T (\check\Lambda_3^N - 2 \check\Lambda_4^N)] \notag \\
&  +  \Theta^T B_0^T ( \check{S}^N - \check\Lambda_1^N - \check\Lambda_2^N - \check\Lambda_2^{NT} - \check\Lambda_4^N ) B_0 \Theta ,  \notag\\
 \check{g}_{11}^N  := &  G^T [ \check\Lambda_{11}^N
+ \check\Lambda_{12}^N(N-1)/N ]/N
 + \Theta^T B_0^T [ \check\Lambda_{11}^N + \check\Lambda_{12}^N (N-1)/N ] B_0 (\Theta + \Theta_1 ) /N  \notag \\
&\quad +  (\check\Lambda_2^N B\Theta_1 + \Theta^T B_0^T
 \check{S}^N B_0 \Theta_1)/N , \notag \\
%\end{align}
%\begin{align}
  \check{g}_{12}^N
  := & [(2 \check\Lambda_4^N - \check\Lambda_3^N ) B\Theta_1
 - G^T \check\Lambda_{12}^N
- \Theta^T B_0^T \check\Lambda_{12}^N B_0 (\Theta + \Theta_1)
+ \Theta^T B_1^T \check\Lambda_3^N B_1 \Theta_1]/N \notag \\
& + \Theta^T B_0^T ( \check{S}^N - \check\Lambda_1^N - \check\Lambda_2^N - \check\Lambda_2^{NT} - \check\Lambda_4^N ) B_0 \Theta_1 , \notag\\
 \check{g}_{22}^N
 := & ( \check\Lambda_{12}^{NT} B \Theta_1 + \Theta_1^T B^T \check\Lambda_{12}^N  ) /N
 + \Theta_1^T B_1^T \check\Lambda_3^N B_1 \Theta_1 (N-1)/N^2
\notag \\
&  - [ (\Theta + \Theta_1)^T B_0^T \check\Lambda_{12}^{NT} B_0 \Theta_1   + \Theta_1^T B_0^T \check\Lambda_{12}^N B_0 (\Theta + \Theta_1) ] /N  \notag \\
 &  + \Theta_1^T B_0^T ( \check{S}^N - \check\Lambda_1^N - \check\Lambda_2^N - \check\Lambda_2^{NT} - \check\Lambda_4^N ) B_0 \Theta_1 , \notag\\
\check{S}^N = &
\check{\Lambda}_1^N + ( \check\Lambda_2^N + \check\Lambda_2^{NT} )(N-1)/N    %\notag\\
  %&
+  \check{\Lambda}_3^N (N-1)/N^2
  + \check{\Lambda}_4^N (N-1)(N-2)/N^2 . \notag
\end{align}

\subsection{Perturbation terms  used in \eqref{ODELam2dN}--\eqref{ODELam22dN}}
\begin{align}
g_1^{bN} = &  \Lambda_1^{bN} B [ ( \mathcal{R}_1(\Lambda_1^{bN} ) + B_0^T S^{bN} B_0/N^2 )^{-1} - (\mathcal{R}_1(\Lambda_1^N) )^{-1} ] B^T \Lambda_1^{bN}
 \notag \\
& - (\Lambda_1^{bN} G + G^T \Lambda_1^{bN}  )/N  %\notag \\
 - (\Lambda_2^{bN} G + G^T \Lambda_2^{bN} ) (N-1)/N^2 \notag \\
& - ( \Gamma^T Q \Gamma/N - \Gamma^T Q - Q \Gamma )/N , \notag
\end{align}
\begin{align}
g_2^{bN} = &  - \Lambda_1^{bN} B [ ( \mathcal{R}_1 ( \Lambda_1^{bN} ) )^{-1} - ( \mathcal{R}_1 ( \Lambda_1^{bN} ) + B_0^T S^{bN} B_0/N^2 )^{-1} ] B^T \Lambda_2^N \notag \\
& - \Lambda_1^{bN} B ( \mathcal{R}_1( \Lambda_1^{bN} )  + B_0^T S^{bN} B_0/N^2 )^{-1} B_0^T S^{bN} B_0 \Theta/N \notag \\
& - ( G^T\Lambda_2^{bN} - \Lambda_2^{bN} G )/N
- G^T ( \Lambda_3^{bN} +(N-2) \Lambda_4^{bN} ) /N^2
 - \Gamma^T Q \Gamma/N , \notag
\end{align}
\begin{align}
g_3^{bN} = g_4^{bN} =  &  - \Theta^T B_0^T [ S^{bN} - \Lambda_1^{bN} - \Lambda_2^{bN} - (\Lambda_2^{bN})^T - \Lambda_4^{bN} ] B_0 \Theta \notag\\
& - (\Lambda_2^{bN})^T B ( \mathcal{R}_1(\Lambda_1^{bN}) )^{-1} B^T \Lambda_2^{bN} \notag \\
& +  [ B^T \Lambda_2^{bN} - B_0^T S^{bN} B_0\Theta/N]^T
 ( \mathcal{R}_1 (\Lambda_1^{bN} ) + B_0^T S^{bN} B_0 /N^2)^{-1} \cdot \notag \\
 &\quad [ B^T \Lambda_2^{bN} - B_0^T S^{bN} B_0\Theta/N]\notag\\
&  - (\Lambda_3^{bN} - 2 \Lambda_4^{bN}) G/N
 - G^T(\Lambda_3^{bN} - 2 \Lambda_4^{bN} )/N, \notag
\end{align}
\begin{align}
 g_{11}^{bN} = & - \Lambda_1^{bN} B (\mathcal{R}_1(\Lambda_1^{bN}))^{-1} B^T \Lambda_{11}^{bN} + \Lambda_1^{bN} B [ \mathcal{R}_1(\Lambda_1^{bN}) + B_0^T S^{bN} B_0/N^2 ]^{-1} \cdot \notag \\
 &\quad [B^T \Lambda_{11}^{bN} - B_0^T(\Lambda_{11}^{bN} + \Lambda_{12}^{bN}(N-1)/N )B_0(\Theta+ \Theta_1)/N  \notag\\
&\qquad - B_0^T S^{bN} B_0 \Theta_1 (N-1)/N^2 ] \notag \\
& - G^T (\Lambda_{11}^{bN}/N + \Lambda_{12}^{bN}(N-1)/N^2)
- \Lambda_2^{bN} B \Theta_1 /N, \notag
\end{align}
\begin{align}
g_{12}^{bN} = & - \Theta^T B_0^T [S^{bN} (N-1)/N - \Lambda_1^{bN} - \Lambda_2^{bN} - (\Lambda_2^{bN})^T - \Lambda_4^{bN}  ] B_0
\Theta \notag \\
& - (\Lambda_2^{bN})^T B \mathcal{R}_1(\Lambda_1^{bN})^{-1} B^T \Lambda_{11}^{bN} \notag \\
& +  [B^T \Lambda_2^{bN} - \Theta^T B_0^T S^{bN} B_0 /N  ]^T
[\mathcal{R}_1(\Lambda_1^{bN}) + B_0^T S^{bN} B_0/N^2 ]^{-1} \cdot \notag \\
& \quad [ B^T \Lambda_{11}^{bN} - B_0^T S^{bN} B_0 \Theta_1 (N-1)/N^2
\notag\\
&\qquad - B_0^T ( \Lambda_{11}^{bN} + \Lambda_{12}^{bN}(N-1)/N ) B_0 (\Theta + \Theta_1 )/N ] \notag \\
& + [ G^T \Lambda_{12}^{bN} + \Theta^T B_0^T \Lambda_{12}^{bN} B_0 (\Theta + \Theta_1) + ( \Lambda_3^{bN} - 2 \Lambda_4^{bN} ) B \Theta_1\notag\\
&\qquad - \Theta^T B_1^T \Lambda_3^{bN} B_1 \Theta_1  ]/N ,
\notag
\end{align}
\begin{align}
g_{22}^{bN} = &\ [ B^T \Lambda_{11}^{bN} - B_0^T S^{bN} B_0 \Theta_1 (N-1)/N^2 \notag \\
& \qquad- B_0^T ( \Lambda_{11}^{bN} + \Lambda_{12}^{bN} (N-1)/N ) B_0 (\Theta + \Theta_1 )/N ]^T \cdot \notag \\
&\quad [\mathcal{R}_1(\Lambda_1^{bN}) + B_0^T S^{bN} B_0 /N^2 ]^{-1}
  [ B^T \Lambda_{11}^{bN} - B_0^T S^{bN} B_0 \Theta_1 (N-1)/N^2 \notag \\
&\qquad - B_0^T ( \Lambda_{11}^{bN} + \Lambda_{12}^{bN} (N-1)/N ) B_0 (\Theta + \Theta_1 )/N ] \notag\\
 &- (\Lambda_{11}^{Nb})^T B ( \mathcal{R}_1(\Lambda_1^{bN}) )^{-1}
B^T \Lambda_{11}^{bN}
 \notag \\
& - \Theta_1^T B_0^T [S^{bN}(N-1)^2/N^2 - \Lambda_1^{bN} - \Lambda_2^{bN} - (\Lambda_2^{bN})^T - \Lambda_4^{bN} ] B_0 \Theta_1\notag\\
& - \Theta_1^T B_1^T \Lambda_3^{bN} B_1 \Theta_1 (N-1)/N^2
\notag \\
& - \Theta_1^T B_0^T [ -\Lambda_{11}^{bN}/N + \Lambda_{12}^{bN}(1-2N)/N^2 ] B_0 (\Theta + \Theta_1 )\notag\\
&- ( \Theta + \Theta_1)^T B_0^T [ -\Lambda_{11}^{bN}/N
+ \Lambda_{12}^{bN}(1-2N)/N^2 ] B_0 \Theta_1  \notag \\
&   - [ (\Lambda_{12}^{bN})^T B \Theta_1 + \Theta_1 B^T \Lambda_{12}^{bN} ] / N, \notag \\
S^{bN} = &  \Lambda_1^{bN} + ( \Lambda_2^{bN} + (\Lambda_2^{bN})^T )(N-1)/N
 +  \Lambda_3^{bN} (N-1)/N^2\notag\\
&  + \Lambda_4^{bN} (N-1)(N-2)/N^2 .  \notag
\end{align}

%%%%%%%%%%%
\section{A limit ODE system}
\label{appendix:PfsolLamdN}

We introduce the following ODE system:
\begin{align}
 & \begin{cases}
  \dot{\Lambda}_1^b =   \Lambda_1^b B
 ( \mathcal{R}_1 ( \Lambda_1^b )  )^{-1} B^T \Lambda_1^b
 - \Lambda_1^b A  -  A^T \Lambda_1^b - Q  ,  \\
%%%
\Lambda_1^b(T) =    Q_f  ,  \quad
 \mathcal{R}_1 ( \Lambda_1^b(t) ) > 0 , \quad \forall t\in[0, T] ,
\end{cases} \label{ODELam1d} \\
%%%%%%%%%%%%%
 & \begin{cases}
  \dot\Lambda_2^b  =  \Lambda_1^b B
 ( \mathcal{R}_1 ( \Lambda_1^b ) )^{-1} B^T\Lambda_2^b
 - ( \Lambda_1^b +  \Lambda_2^b ) G  \\
 \hspace{1.1cm}
 - A^T  \Lambda_2^b - \Lambda_2^b  ( A -   B \Theta )
 + Q \Gamma  , \\
%%%
\Lambda_2^b(T) =  -  Q_f \Gamma_f ,
\end{cases} \label{ODELam2d}
\end{align}
%%%%%%%%%%%%%
\begin{align}
& \begin{cases}
\label{ODELam3d}
  \dot\Lambda_3^b =    \Lambda_2^{bT}
B ( \mathcal{R}_1( \Lambda_1^b ) )^{-1} B^T \Lambda_2^b
 - (  \Lambda_2^b + \Lambda_4^b  )^T G \\
\hspace{1.1cm}
 - G^T (\Lambda_2^b + \Lambda_4^b )
   - \Lambda_3^b  (A- B \Theta)
 - ( A - B \Theta )^T \Lambda_3^b \\
\hspace{1.1cm}
 - \Theta^T B_0^T ( \Lambda_1^b  + \Lambda_2^b
 + \Lambda_2^{b T}  + \Lambda_4^b ) B_0   \Theta     \\
\hspace{1.1cm}
 - \Theta^T  B_1^T \Lambda_3^b  B_1 \Theta
 - \Gamma^T Q \Gamma , \\
%%%
\Lambda_3^b (T)  =   \Gamma_f^T Q_f \Gamma_f   ,
\end{cases}
\end{align}
\begin{align}
%%%%%%%%
& \begin{cases}
  \dot\Lambda_4^b  =   \Lambda_2^{bT}
 B ( \mathcal{R}_1( \Lambda_1^b  ) )^{-1} B^T \Lambda_2^b
 - \Lambda_2^{bT} G - G^T \Lambda_2^b  \\
\hspace{1.1cm}
- \Lambda_4^b  (A+G -  B \Theta )
- (A+G -  B \Theta )^T  \Lambda_4^b  \\
\hspace{1.1cm}
 - \Theta^T B_0^T (\Lambda_1^b + \Lambda_2^b
 +  \Lambda_2^{bT} + \Lambda_4^b  ) B_0 \Theta
 - \Gamma^T Q \Gamma  ,  \\
%%%%
 \Lambda_4^b  (T) =  \Gamma_f^T Q_f \Gamma_f   ,
\end{cases} \label{ODELam4d} \\
%%%%%%%%%%%%%
& \begin{cases}
  \dot\Lambda_{11}^b =   \Lambda_1^b   B
 ( \mathcal{R}_1 (\Lambda_1^b  ) )^{-1} B^T \Lambda_{11}^b
 + \Lambda_2^b  B \Theta_1 - A^T \Lambda_{11}^b  \\
\hspace{1.1cm}
- \Lambda_{11}^b (A+G - B( \Theta + \Theta_1 ) )   ,\\
%%%
\Lambda_{11}^b (T) = 0 ,
\end{cases} \label{ODELam11d}
\end{align}
\begin{align}
 & \begin{cases}
  \dot\Lambda_{12}^b  =   \Lambda_2^{bT} B
 ( \mathcal{R}_1  (\Lambda_1^b ) )^{-1} B^T \Lambda_{11}^b
 - G^T ( \Lambda_{11}^b + \Lambda_{12}^b  )   \\
 \hspace{1.1cm}
 - ( A  - B \Theta )^T \Lambda_{12}^b
 - \Lambda_{12}^b  (A+G - B(\Theta + \Theta_1 ) ) \\
 \hspace{1.1cm}
 - \Theta^T B_0^T
(\Lambda_1^b  + \Lambda_2^b  + \Lambda_2^{bT}
+ \Lambda_4^b  ) B_0 \Theta_1  \\
 \hspace{1.1cm}
 - \Theta^T B_0^T ( \Lambda_{11}^b   + \Lambda_{12}^b  ) B_0 (\Theta + \Theta_1 ) + \Lambda_4^b  B \Theta_1  , \\
%%%
\Lambda_{12}^b(T) = 0 ,
\end{cases}\label{ODELam12d}
\end{align}
\begin{align}
 & \begin{cases}
   \dot{\Lambda}_{22}^b
  =    \Lambda_{11}^{bT} B
 ( \mathcal{R}_1( \Lambda_1^b ) )^{-1} B^T \Lambda_{11}^b
\\
\hspace{1.1cm}
  - \Lambda_{22}^b  (A+G - B ( \Theta + \Theta_1 ) )
  + \Lambda_{12}^{bT} B \Theta_1  \\
\hspace{1.1cm}
  -  (A+G - B( \Theta + \Theta_1 ) )^T \Lambda_{22}^b
 + \Theta_1^T B^T \Lambda_{12}^b   \\
\hspace{1.1cm}
  - \Theta_1^T B_0^T ( \Lambda_1^b  + \Lambda_2^b
 + \Lambda_2^{bT} + \Lambda_4^b  ) B_0 \Theta_1 \\
\hspace{1.1cm}
 - ( \Theta + \Theta_1 )^T B_0^T \Lambda_{22}^b  B_0
 ( \Theta + \Theta_1 )   \\
\hspace{1.1cm}
- \Theta_1^T B_0^T ( \Lambda_{11}^b  + \Lambda_{12}^b  ) B_0
 (\Theta + \Theta_1 ) \\
\hspace{1.1cm}
 - (\Theta + \Theta_1 )^T B_0^T
(  \Lambda_{11}^b   + \Lambda_{12}^b  )^T B_0 \Theta_1   ,  \\
%%%
\Lambda_{22}^b (T) = 0 .
\end{cases} \label{ODELam22d}
\end{align}
Under Assumption~\ref{assm:solLam}, the coefficients in \eqref{ODELam1d}--\eqref{ODELam22d} are defined on $[0, T]$. We may regard \eqref{ODELam1d}--\eqref{ODELam22d}
as the limit ODE system for \eqref{ODELam1dN}--\eqref{ODELam22dN}.
\begin{lemma}\label{lemma:Lam1d}
Under Assumption 1, the ODE system \eqref{ODELam1d}--\eqref{ODELam22d} admits a unique solution on $[0,T]$.
\end{lemma}

\begin{proof}
We  have that \eqref{ODELam1d} admits a unique solution
$\Lambda_1^b =\Lambda_1$ on $[0, T]$.
With $\Lambda_1^b$ obtained from solving \eqref{ODELam1d}, \eqref{ODELam2d} is a first order linear ODE and admits a unique solution
$\Lambda_2^b$ on $[0, T]$.
Given $(\Lambda_1^b, \Lambda_2^b)$ on $[0, T]$, the ODE system \eqref{ODELam3d}--\eqref{ODELam12d} is a first order linear ODE system and admits a unique solution $(\Lambda_3^b, \cdots, \Lambda_{12}^b)$ on $[0, T]$. Finally, we further uniquely solve \eqref{ODELam22d} on $[0, T]$.
\end{proof}

%%%%%%%%%%%%%%%%%%%%%
\bibliographystyle{abbrv}
\bibliography{HYMFGRef}

\end{document}